\documentclass[10pt, a4paper, twoside,reqno]{amsart}

\usepackage{verbatim}

\usepackage{fancyhdr}
\usepackage{mathabx}

\makeatletter

\def\section{\@startsection{section}{1}%
	\z@{.7\linespacing\@plus\linespacing}{.5\linespacing}%
	{\normalfont \Large\scshape\centering}}

\def\subsection{\@startsection{subsection}{2}%
	\z@{.5\linespacing\@plus.7\linespacing}{.5\linespacing}%
	{\normalfont\large\bfseries}}

\def\subsubsection{\@startsection{subsubsection}{3}%
	\z@{.5\linespacing\@plus.7\linespacing}{.5\linespacing}%
	{\normalfont\itshape}}

\usepackage[usenames, dvipsnames]{color}
\definecolor{darkblue}{rgb}{0.0, 0.0, 0.45}

\usepackage[colorlinks	= true,
raiselinks	= true,
linkcolor	= darkblue, 
citecolor	= Mahogany,
urlcolor	= ForestGreen,
pdfauthor	= {Georgios Darivianakis},
pdftitle	= {},
pdfkeywords	= {},
pdfsubject	= {},
plainpages	= false]{hyperref}

\allowdisplaybreaks
\date{\today}

\usepackage{graphicx, psfrag}
\usepackage{amsthm}
\usepackage{amsmath} 
\usepackage{amssymb}  
\usepackage[noadjust]{cite}
\usepackage{color}
\usepackage{enumerate}
\usepackage{subfigure}
\usepackage{multirow}
\usepackage{booktabs}
\usepackage{subfigure}
\graphicspath{{drawing/}}

\usepackage{tikz}
\usetikzlibrary{calc,patterns,decorations.pathmorphing,decorations.markings}

\usepackage{xspace}

\newtheorem{proposition}{Proposition}
\newtheorem{lemma}{Lemma}
\newtheorem{theorem}{Theorem}

\newtheorem{example}{Example}
\newtheorem{corollary}{Corollary}

\newcommand{\st}{{\textrm{s.t.}}}
\newcommand{\R}{{\mathbb R}}

\newcommand{\mb}{\mathbb}
\newcommand{\mc}{\mathcal}

\newcommand{\bs}{\boldsymbol}

\newcommand{\ol}[1]{\overline{#1}}
\newcommand{\wt}[1]{\widetilde{#1}}
\newcommand{\wh}[1]{\widehat{#1}}

\newcommand{\Ni}{{\scriptscriptstyle \mathcal N_i \scriptstyle}}
\newcommand{\oNi}{{\scriptscriptstyle \overline{\mathcal N}_i \scriptstyle}}

\newcommand{\oNj}{{\scriptscriptstyle \overline{\mathcal N}_j \scriptstyle}}
\newcommand{\Mu}{{\scriptscriptstyle \mathcal M \scriptstyle}}
\newcommand{\AT}{{\scriptscriptstyle \mc{AT} \scriptstyle}}
\newcommand{\CC}{{\scriptscriptstyle \mc{CC} \scriptstyle}}
\newcommand{\FS}{{\scriptscriptstyle \mc{FS} \scriptstyle}}

\newcommand{\PCd}{Problem~$(\text{C}:\mc {SC}(\bs \xi_\Mu))$\xspace}

\newcommand{\PSCs}{Problem~\eqref{Semi-Centralized}\xspace}
\newcommand{\PSCd}{Problem~$(\text{PN}:\mc {SC}(\bs \xi_\oNi))$\xspace}

\newcommand{\PDs}{Problem~\eqref{Decentralized}\xspace}
\newcommand{\PDd}{Problem~$(\text{L}:\mc {SC}(\bs \xi_i, \bs \zeta_{\Ni}), \mc F(\mb R^{N_{x,i}}))$\xspace}
\newcommand{\PDdCC}{Problem~$(\text{L}:\mc {SC}(\bs \xi_i, \bs \zeta_{\Ni}), \mc F_{\CC}(\mb R^{N_{x,i}}))$\xspace}
\newcommand{\PDdCCa}{Problem~$(\text{L}:\mc {SC}_a(\bs \xi_i, \bs \zeta_{\Ni}), \mc F_{\CC}(\mb R^{N_{x,i}}))$\xspace}
\newcommand{\PDdAT}{Problem~\eqref{DecentralizedXab}\xspace}
\newcommand{\PDdATC}{Problem~\eqref{DecentralizedFinal}\xspace}

\title[Decentralized decision making for networks of uncertain systems]{Decentralized decision making for networks of uncertain systems}

\author{Georgios Darivianakis, Angelos Georghiou and John Lygeros
	}%
	\thanks{This research project is financially supported by the Swiss Innovation Agency Innosuisse and is part of the Swiss Competence Center for Energy Research SCCER FEEB\&D.} 
	\thanks{The authors are with the Automatic Control Laboratory, Department of Electrical Engineering and Information Technology, ETH Zurich, CH-8092 Zurich, Switzerland, and also with the Desautels Faculty of Management, McGill University, Montreal, QC H3A 1G5, Canada (e-mail: gdarivia@control.ee.ethz.ch; angelos.georghiou@mcgill.ca; jlygeros@control.ee.ethz.ch).}
	
\begin{document} 

\begin{abstract}
Distributed model predictive control (MPC) has been proven a successful method in regulating the operation of large-scale networks of constrained dynamical systems. This paper is concerned with cooperative distributed MPC in which the decision actions of the systems are usually derived by the solution of a system-wide optimization problem. However, formulating and solving such large-scale optimization problems is often a hard task which requires extensive information communication among the individual systems and fails to address privacy concerns in the network. Hence, the main challenge is to design decision policies with a prescribed structure so that the resulting system-wide optimization problem to admit a loosely coupled structure and be amendable to distributed computation algorithms. In this paper, we propose a decentralized problem synthesis scheme which only requires each system to communicate sets which bound its states evolution to neighboring systems. The proposed method alleviates concerns on privacy since this limited communication scheme does not reveal the exact characteristics of the dynamics within each system. In addition, it enables a distributed computation of the solution, making our method highly scalable. We demonstrate in a number of numerical studies, inspired by engineering and finance, the efficacy of the proposed approach which leads to solutions that closely approximate those obtained by the centralized formulation only at a fraction of the computational effort.
\end{abstract}

\maketitle

\section{Introduction}

Operation of large-scale networks of interacting dynamical systems remains an active field of research due to its high impact on real-world applications, e.g., regulation of power networks \cite{Venkat2008} and energy management of building districts \cite{Darivianakis2016}. For system of this scale, the design and deployment of a centralized controller is often difficult due to computation and communication limitations, and also prohibitive in cases the individual systems need to retain a certain degree of privacy. In such cases, it is desirable the design of interacting local controllers which rely only on local computational resources and a distributed communication network with a prescribed structure.

The problem of synthesizing optimal distributed controllers based on an arbitrary communication structure typically amounts to an infinite-dimensional, non-convex optimization problem and is a known NP-hard \cite{Tsitsiklis1985}. For that reason, several studies have been devoted to identifying communication structures under which the problem of designing optimal decentralized controllers can be simplified \cite{Lin2011,Mahajan2012}. For instance, if the communication network admits a partially nested structure \cite{Ho1972} then affine controllers are known to be optimal for decentralized linear systems with quadratic costs and additive Gaussian noise \cite{Ho1972,Rantzer2006,Rantzer2006a}. Similar results exist for communication structures that are spatially invariant \cite{Fardad2009,Bamieh2002,Motee2008}, introduce delays on information sharing \cite{Lamperski2015,Nayyar2011,Nayyar2013}. Recent advances with convex optimization algorithms shifted research interest on identifying information structures that allow the optimal distributed controllers synthesis problem to be formulated as a convex optimization problem \cite{Bamieh2005,DeCastro2002,Dvijotham2013,Matni2013,Qi2004}. These structures usually possess properties as quadratic invariance \cite{Rotkowitz2006,Swigart2014} and funnel causality \cite{Bamieh2005} which essentially eliminate the incentive of signaling among the decentralized controllers. For general network structures, the usual practice is to resort to linear matrix inequality relaxations \cite{Langbort2004,Zecevic2010} or semidefinite programming relaxations \cite{Lavaei2012,Fazelnia2017} to obtain a suboptimal design with performance guarantees.

A downside of the aforementioned approaches is the inability of the resulting static controllers to cope with state and input constraints in the systems. MPC is an optimization based methodology that is well-suited for regulating the operation of constrained systems \cite{Mayne2000}. In the MPC framework adopted here, distributed control schemes are usually categorized into cooperative or non-cooperative \cite{Scattolini2009}. Cooperative distributed MPC approaches require substantial communication infrastructure and computation resources since a system-wide MPC problem is formulated and solved \cite{Venkat2005,Stewart2010,Giselsson2014}. On the other hand, non-cooperative approaches, though computationally simple and effective in practice, can be conservative in presence of strong coupling \cite{Richards2004,Keviczky2006,Trodden2010}. Both cooperative and non-cooperative schemes typically require a centralized offline design phase. This requirement can be restrictive, making the distributed schemes suffer from similar complexity and privacy concerns as the centralized MPC formulation. To alleviate these issues, it is desirable to develop decentralized schemes that rely on local computational resources and information structure. In the literature, this is commonly achieved by each system considering the worst-case effect of its neighbors as a bounded exogenous disturbance to its own system \cite{Camponogara2002,Dunbar2007,Farina2012,Lucia2015,Trodden2017}. Nevertheless, this can be conservative approach if the sets of bounded exogenous disturbances are calculated offline; therefore, disregarding the possibility of adapting their size based on the dynamical evolution of the system.

In this paper, we consider the problem of designing optimal decentralized cooperative MPC controllers for linear time varying interconnected systems. We assume that these distributed systems are coupled through their dynamics and/or constraints and are subject to possibly correlated exogenous disturbances which are assumed to have known and bounded support. Unless a nested information structure is imposed, as suggested in \cite{Lin2016}, the aforementioned synthesis problem is computationally intractable. In this paper, however, we are interested on minimum information structures which only require communication among neighboring systems and hence do not necessarily admit any nestedness. In particular, each system is only required to transmit to its direct neighbors a set that bounds these of its predicted states which affect their constraints and/or dynamics. Due to this minimal communication structure, the proposed decentralized scheme allows the decoupling of the optimization problems of the individual systems in the network. In this setting, however, we abandon the search for optimal decentralized control policies and resort, instead, to approximation of the original non-convex, infinite dimensional problem. We relax NP-hardness of the original formulation by restricting ourselves on communicated sets that result as the scaling and translation of a predefined convex conic set. The proposed method scales polynomially with respect to the number of agents and the prediction horizon length. The polynomial scalability is achieved by reformulating the original problem into a convex infinite dimensional optimization problem, and then approximating it using decision rules \cite{bental2004ars}. The resulting problem retains its decoupled structure making it amendable to distributed computations algorithms such as the alternating direction method of multipliers \cite{Boyd2011}. The proposed method partly alleviates concerns on privacy by not revealing sensitive information regarding the operational characteristics of the individual agents. The importance of adaptive set communication as a method for decentralized cooperative MPC has also been investigated in recent works (e.g., \cite{Farina2012, Lucia2015, Trodden2017}). The key difference of our approach is that the size of these communicated sets is adapted online being a decision variable in the resulting optimization problem. A proof-of-concept study for the problem of the efficient energy management of building districts was presented in \cite{DarivianakisDec}. This paper considerably extends the content of our preliminary paper by providing a mathematically rigorous presentation of the proposed method and by investigating its merits and demerits through a number of illustrative examples inspired from engineering and finance.

The remainder of this paper is organized as follows. Section \ref{sec::ProbForm} provides the problem formulation and briefly reviews available methods for the design of optimal decision policies with centralized and nested information structures. The main contributions are presented in Sections \ref{sec::DecCont} and \ref{sec::SolMethod}, where the proposed approach is developed and the techniques associated with the derivation of a tractable approximation to the original non-convex infinite-dimensional problem are discussed. Section \ref{sec::Numerics} provides numerical studies to demonstrate the efficacy and scalability of the proposed method. Section \ref{sec::Conclusion} closes the paper with conclusions and possible directions for future work. Proofs of the propositions and theorems are found in the Appendix.

\textbf{Notation:} For given vectors $ v_{i} \in \mb R^{k_i} $ with $ k_i \in \mb N $, $ i \in \mc M = \{1,\ldots,m\} $, we define $ v_\Mu = [v_{i}]_{i\in \mc M} = [v_{1}^\top \ldots v_{m}^\top]^\top \in \mb R^{k} $ with $ k = \sum_{i=1}^{m}k_i $ as their vector concatenation. Concatenated vectors are represented in boldface. Dimensions of matrices and concatenated vectors are assumed clear from the context. We denote by $ t $ the time instant of a horizon $ T \in \mb N $, and we define the sets $ \mc T = \{1,\ldots, T\} $ and $ \mc T_+ = \mc T \cup \{T+1\} $. Given time dependent vectors $ \nu_{i,t} \in \mb R^{\ell_i} $ with $ i \in \mc M $, $ t \in \mc T $ and $ \ell_i \in \mb N $, we define $ \bs \nu_{\Mu,t} = [\nu_{i,t}]_{i \in \mc M} $ as the concatenated vector at time $t$, $ \bs \nu_i^t = [\nu_{i,1}^\top \ldots \nu_{i,t}^\top]^\top $ as the history of the $ i $-th vector up to time $ t $, and $ \bs \nu_\Mu^t = [\bs \nu_i^t]_{i \in \mc M} $ as the history of the concatenated vector up to time $ t $.

\section{Problem formulation} \label{sec::ProbForm}

We consider a physical network comprising $ M $ interconnected systems, henceforth referred to as agents. We assume that the agents are coupled among themselves through the dynamics. We describe these interactions through a directed graph in which an arc connecting agent $ j $ to agent $ i $, with $ i, j\in \mc M = \{1,\ldots,M\} $, indicates that the states of the $ j $-th agent affect the dynamics of the $ i $-th agent. We refer agent $ j $ as the preceding neighbor to agent $ i $, henceforth neighbor, and we define the set $ \mc N_i \subset \mc M $ to include all the neighbors of the $ i $-th agent. Fig. \ref{fig::physNet} illustrates a system of $ M = 5 $ agents where the neighbors of agent 3 are given by $ \mc N_3 = \{2, 5\}$. In the sequel, we refer to the {physical network} depicted in Fig. \ref{fig::physNet} as the ``working example'' and  use it to streamline the presentation of key ideas in the paper. 
\begin{figure}[h]
	\centering
	\includegraphics[width = 0.4\textwidth]{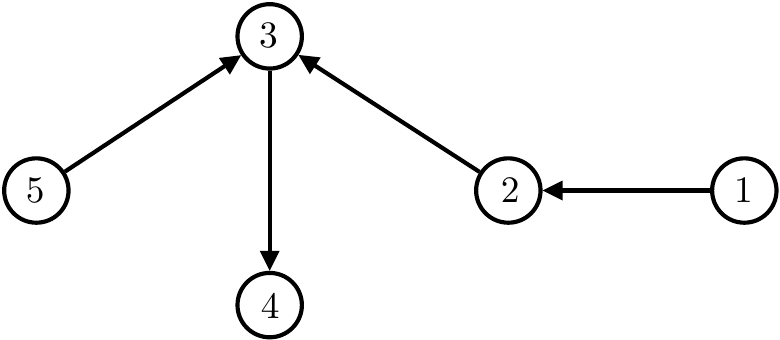}
	\caption{\textbf{Working example}: \emph{Physical coupling graph of 5 agents}. Solid line arrows represented the direction of the interaction.}
	\label{fig::physNet}
\end{figure}

\subsection{System dynamics, constraints and objective function} 

In this paper, we study finite horizon problems with $T$ stages. We use linear dynamics to model the state evolution of the agent $i$ at time instant $ t \in \mc T $, as 
\begin{equation} \label{eq::stateDynamics}
x_{i,t+1} = A_{i,t} x_{i,t} + B_{i,t} \bs x_{\Ni,t} +D_{i,t} u_{i,t} + E_{i,t}\xi_{i,t}. 
\end{equation} 
Here $ x_{i,t} \in \mb R^{n_{x,i}} $ denotes the states, with the initial state $ x_{i,1} $ known. The interaction of agent $i$ and its neighbors is captured through the term $ B_{i,t} \bs x_{\Ni,t}$.
The vector $ u_{i,t} \in \mb R^{n_{u,i}} $ models the inputs and $ \xi_{i,t} \in \mb R^{n_{\xi,i}} $ captures the exogenous disturbances affecting the system dynamics. The time-varying system matrices $ A_{i,t} $, $ B_{i,t} $, $D_{i,t} $ and $ E_{i,t} $ are assumed known with proper dimensions and of full column rank. To economize on notation, we now compactly rewrite \eqref{eq::stateDynamics} as 
\begin{equation*}\label{eq::stateDynamicsCompact}
\bs x_i = f_i(\bs x_{\Ni}, \bs u_i, \bs \xi_i) := A_i x_{i,1} + B_i \bs x_{\Ni} + D_i \bs u_i + E_i \bs \xi_i,
\end{equation*}
where $ \bs x_i = [x_{i,t}]_{t\in \mc T_+} $, $ \bs u_i = [u_{i,t}]_{t \in \mc T} $, $ \bs \xi_i = [\xi_{i,t}]_{t\in \mc T} $ and $ \bs x_{\Ni} = [\bs x_{\Ni,t}]_{t \in \mc T} $. The system matrices $ A_i $, $ B_i $, $ D_i $ and $ E_i $ used to define the function $ f_{i}(\cdot)  $ are directly constructed by the problem data given in \eqref{eq::stateDynamics} (see e.g., \cite{Goulart2006} for such a derivation). 
The  $ i $-th agent is subject to linear operational constraints
\begin{equation}\label{eq::InequalitiesCompact}
\mc O_i = \big\{(\bs x_{i}, \bs u_{i}) \,:\, H_{x,i}\bs{x}_i + H_{u,i} \bs u_i \leq h_i\big\},
\end{equation}
where the matrices $ H_{x,i} $, $ H_{u,i} $ and $ h_i $ are assumed known and of proper dimensions. Note that this compact constraint formulation allows the consideration of time-varying linear operational constraints with time-stage coupling. In addition, operational constraints involving neighboring states or exogenous disturbances can always be included by appropriately extending the state space of the $ i $-th system.

The $ i $-th agent's objective function is given as,
\begin{equation}\label{eq::objFnc}
J_i(\bs x_i,\bs u_i) =  \sum_{t=1}^T \left(\| Q_i  x_{i,t} \|_p + \| R_i  u_{i,t} \|_p\right), 
\end{equation}
where $ p \in \{\infty, 1, 2\} $ allows for different objective formulation. The penalization matrices $  Q_i $, $  R_i  $ are assumed known and of proper dimensions.

In the following, we assume that the exogenous disturbances affecting agent $ i $ reside in the nonempty, convex and compact polyhedral uncertainty set $\Xi_i = \{\bs \xi_i :  W\bs \xi_i \geq  w\}$ where matrix $W$ and  vector $w$ are known and of proper dimensions. We will be making the simplifying assumption that the joint uncertainty set of all agents in the system has a decoupled structure, i.e.,
$ \bs \xi_\Mu \in \Xi_\Mu = \bigtimes_{i \in \mc M} \Xi_i $, which  essentially precludes the existence of  coupled disturbances amongst the agents. This assumption can be  relaxed at the expense of further case distinctions in what follows.


\subsection{Designing decision policies with centralized information exchange}

A common assumption in designing decision policies is to assume that at time $t$, each agent has access to the states of all the other agents in the network up to and including period $t$ \cite{HGK:2011b,Goulart2006}. We will refer to this communication as the \emph{centralized information exchange}, depicted in Fig. \ref{fig::InfStr_CS} for the working example. In this context, we denote the \emph{causal state feedback} policies for agent $i$ at time $t\in\mc T$ as  \mbox{$\pi_{i,t}:\mb R^{n_x^t} \rightarrow \mb R^{n_{u,i}} $} where $ n_{x}^t = t \left(\sum_{j\in \mc M} n_{x,j} \right)$, such that its input at time $ t $ is given as $ u_{i,t} = \pi_{i,t}(\bs x_\Mu^t) $. We write $ \bs \pi_i(\bs x_\Mu)  = [\pi_{i,t}(\bs x_\Mu^t)]_{t\in\mc T} $ to denote the policy concatenation over the horizon, and we define as $ \mc C(\bs x_\Mu) $ the space of causal state feedback policies.

\begin{figure}[h]
	\centering
	\includegraphics[width = 0.4\textwidth]{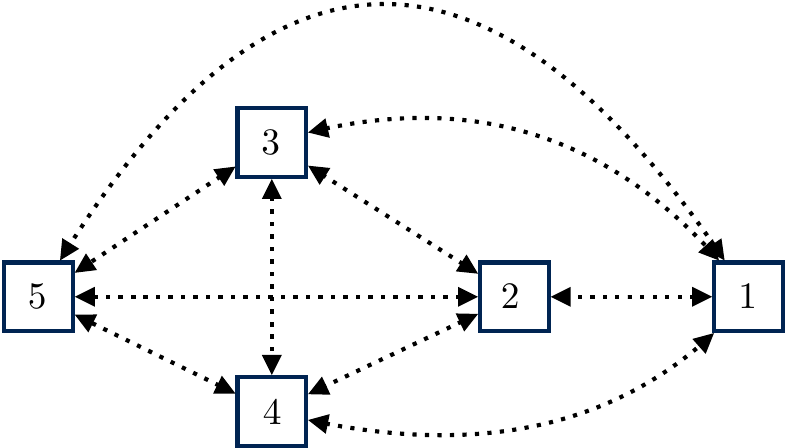}
	\caption{\textbf{Working example}: \emph{Centralized information exchange}. Dotted line arrows represent the  communication links between the agents.}
	\label{fig::InfStr_CS}
\end{figure}

The optimization problem for designing robust policies with centralized information exchange which minimize the sum of worst-case individual objectives is formulated as
\begin{equation}\label{Centralized}
\begin{array}{l}
\text{\;\;minimize } \displaystyle\sum\limits_{i = 1}^M \max\limits_{\bs \xi \in \Xi} J_i(\bs x_i,\bs u_i) \\
\left.\begin{array}{@{}r@{\;}l@{}}
\text{subject to }& \bs \pi_i(\cdot) \in \mc C(\bs x_\Mu), \,
\bs u_i = \bs \pi_i(\bs x_\Mu)\\
& \bs x_{i} = f_{i}\big(\bs x_{\Ni}, \bs u_{i}, \bs \xi_{i}\big)\\
& (\bs x_{i}, \bs u_{i}) \in \mc O_i
\end{array}\right \rbrace \forall \bs \xi_\Mu \in \Xi_\Mu,\; \forall i \in \mc M.
\end{array}
\tag{$\text{C}:\mc C(\bs x_\Mu)$}
\end{equation}
The optimization variables  are $ \bs \pi_i(\cdot) $ for all $ i \in \mc M $. As shown in \cite{HGK:2011b,Goulart2006}, the state feedback structure of the decision variables induces a non-convex feasible region. To deal with this, they propose the design of \emph{strictly causal disturbance feedback policies} $ \Pi_{i,t}:\mb R^{n_{\xi}^t} \rightarrow \mb R^{n_{u,i}} $ where $ n_{\xi}^t = (t-1)\sum_{i\in \mc M} n_{\xi,i} $, such that the input at each time step is given by $ u_{i,t} = \Pi_{i,t}(\bs \xi_\Mu^{t-1}) $. We write $ \bs \Pi_i(\bs \xi_\Mu)  = [\Pi_{i,t}([\bs \xi_\Mu^{t-1}]_{t\in\mc T}) $ to denote the policy concatenation over the horizon, and we define as $ \mc {SC}(\bs \xi_\Mu) $ the space  of strictly causal disturbance feedback policies.
This formulation leads to the infinite dimensional linear optimization \PCd. Using the fact that matrices $A_{i,t}$, $B_{i,t}$ and $E_{i,t}$ are full column rank, they show that there is an   one-to-one relationship between state and disturbance feedback policies in terms of feasibility and optimality. Restricting the admissible policies to have an affine structure reduces the problem to finite-dimensional linear optimization problem which can be solved with off-the-shelf optimization solvers \cite{Goulart2006}. 
Furthermore, due to the one-to-one relationship of state and feedback policies, there is also a unique mapping that translates the resulting affine disturbance  policies to an equivalent affine state feedback policy which  allows  to be implemented locally by each agent using the centralized communication network. 

Although theoretically appealing, decision rules based on centralized information exchange are hard to design and implemented in practice for large-scale systems. This is partially the case since for large networks $ (i) $ solving the resulting  linear optimization problem from the affine approximation  can be computationally challenging due to its monolithic structure; while $ (ii) $ the excessive  communication and the centralized physical network required to allow each agent to evaluate its policy, does not promote privacy since the exact policy/constraints of individual agents are eventually revealed to the rest of the network. We will demonstrate the former through numerical experiments in Section \ref{sec::Numerics}.

\subsection{Designing policies with partially nested information exchange}

In an attempt to address the computational and privacy issues, researchers have proposed a number of policy designs that consider only partial communication among the agents. For an arbitrary information exchange network the design phase typically results in a non-convex problem which is computationally intractable. A notable exception is the work of \cite{Lin2016} which assumes a \emph{partially nested information exchange}, leading to  convex formulations. Roughly speaking this communication exchange implies that  agent $ i $ has access to information coming from all of its \emph{precedent agents} in a non-anticipative manner. In this setting, agent $ j $ is named a\emph{ precedent to agent $ i $}, if input at system $ j $ at time $ t' $ can affect the local information available to agent $ i $ at some time $ t > t' $ in the future \cite[Definition 1]{Lin2016}. The partially nested information exchange associated with the working example is depicted in Fig. \ref{fig::InfStr_NS}.

In a partially nested communication, the policy is designed as follows. We denote by $ \ol{\mc N}_i \subseteq \mc M $ the set that includes agent $ i $ and all its precedent agents. At time $t$,  the $ i $-th agent measures  its own states and the states of its precedent agents, denoted by  $ \bs x_{\oNi,t} = [x_{j,t}]_{j \in \ol{\mc N}_i} $. Using all measurements from stage 1 up to time $t$, it designs a \emph{causal partial state feedback} policy $ \phi_{i,t}:\mb R^{\ol{n}_{x,i}^t} \rightarrow \mb R^{n_{u,i}} $, where $ \ol{n}_{x,i}^t = t \left(\sum_{j\in \ol{\mc N}_i} n_x^j\right) $. The input is now given as $ u_{i,t} = \phi_{i,t}(\bs x_\oNi^t) $. We write $ \bs \phi_i(\bs x_\oNi)  = [\phi_{i,t}(\bs x_\oNi^t)]_{t\in\mc T} $ to denote the policy concatenation over the time horizon, and we define  $ \mc C(\bs x_\oNi) $ the space  of causal state feedback policies associated with agent $i$. The nested information structure associated with the working example is depicted in Fig. \ref{fig::InfStr_NS}.
\begin{figure}[h]
	\centering
	\includegraphics[width = 0.4\textwidth]{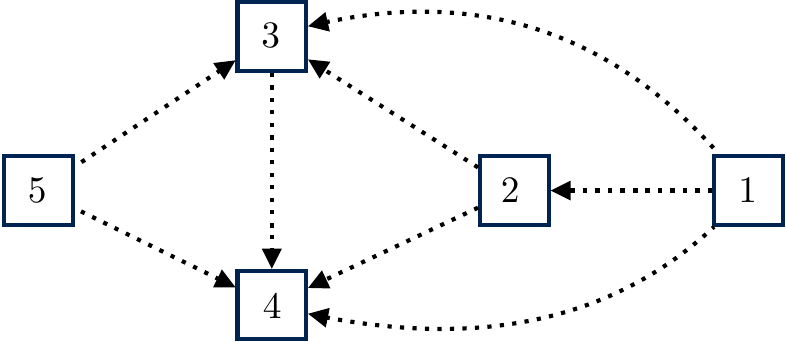}
	\caption{\textbf{Working example}: \emph{Partially nested information exchange}. Dotted line arrows represent the established communication links between local controllers.}
	\label{fig::InfStr_NS}
\end{figure}

The optimization problem to design robust  policies with partially nested information exchange which minimize the sum of worst-case individual objectives is formulated as
\begin{equation}\label{Semi-Centralized}
\begin{array}{l}
\text{\;\;minimize } \displaystyle\sum\limits_{i = 1}^M \max\limits_{\bs \xi_\oNi \in \Xi_\oNi} J_i(\bs x_i,\bs u_i) \\
\left.\begin{array}{@{}r@{\;}l@{}}
\text{subject to }& \bs \phi_{i}(\cdot) \in \mc C(\bs x_\oNi),\;\bs u_i = \bs \phi_{i}(\bs x_\oNi)\\
& \bs x_{i} = f_{i}\big(\bs x_{\Ni}, \bs u_{i}, \bs \xi_{i}\big)\\
& (\bs x_{i}, \bs u_{i}) \in \mc O_i
\end{array}\right \rbrace \forall \bs \xi_\oNi \in \Xi_{\oNi},\;
\forall i \in \mc M,
\end{array}
\tag{$\text{PN}:\mc C(\bs x_\oNi)$}
\end{equation}
where $ \Xi_{\oNi} = \bigtimes_{j \in \ol{\mc N}_i} \Xi_j $. The decision variables are $ \bs \phi_{i}(\cdot) $ for all $ i \in \mc M $. \PSCs is typically non-convex due to the state feedback structure of the policies. Similar to the centralized case, \cite{Lin2016} proposes the use of \emph{strictly causal partial nested disturbance feedback policies} $ \Phi_{i,t}:\mb R^{\ol{n}_{\xi,i}^t} \rightarrow \mb R^{n_{u,i}} $ where $ \ol{n}_{\xi,i}^t = (t-1) \left( \sum_{j\in \ol{\mc N}_i} n_\xi^j\right) $, such that the input at each time step is given by $ u_{i,t} = \Phi_{i,t}(\bs \xi_\oNi^{t-1}) $.
We write $ \bs \Phi_i(\bs \xi_{\oNi})  = [\Phi_{i,t}(\bs \xi_\oNi^{t-1})]_{t\in\mc T} $ to denote the policy concatenation, and we define as $ \mc {SC}(\bs \xi_\oNi) $ the space  of strictly causal disturbance feedback policies, leading to the infinite dimensional linear optimization \PSCd.
If these disturbance feedback policies are restricted to admit an affine structure then \PSCd becomes a finite-dimensional linear optimization problem. As with the centralized case, there is an one-to-one relationship between the state and disturbance feedback policies, both for the infinite dimensional and affine restriction. This  allows the agents in the network to evaluate their policies based on the established  partially nested communication.

The following theorem establishes the connection between centralized and partially nested information design problems, and will be used in the following section  to demonstrate the relationship between the proposed approach to the centralized and partially nested policy structures.
\begin{theorem}\label{thm::1}
	\PSCd is a conservative approximation of \PCd in the following sense: every feasible solution of \PSCd is feasible in \PCd, and the optimal value of \PSCd is  larger or equal to the optimal value of \PCd.
\end{theorem}

Partially nested information exchange slightly reduces the communication requirements compared to the centralized problem (see Figs. \ref{fig::InfStr_CS} and \ref{fig::InfStr_NS} of the working example). This has a positive impact on the solution time needed to design affine feedback policies, however, the resulting linear program inherits in large part the monolithic structure of the centralized problem due to absence of a  non-sparse structure. Most importantly, even in the simple model of our working example, the partially nested communication requires three additional links compared  to the physical links. In some cases, the minimum number of communications links needed to ensure a partially nested information coincides with the centralized information exchange, see the examples presented in Section \ref{sec::Numerics}.
Therefore, the synthesis of policies with a partially nested information  inherit, in large extend, the drawbacks of the centralized problem, both from a computational and privacy standpoint.

\section{Designing decision policies with local information exchange} \label{sec::DecCont}
In this paper we propose a decentralized policy structure that relies on local information exchange among the agents. The proposed policy aims to address both the computational and privacy concerns discussed so far. While in the previous section the information flow had to be sufficiently complex as to capture a partially nested structure, in this section we will assume that the information flow can be as simple as the physical network, as this is illustrated in Fig. \ref{fig::InfStr_DS} for the working example. 
\begin{figure}[h]
	\centering
	\includegraphics[width = 0.4\textwidth]{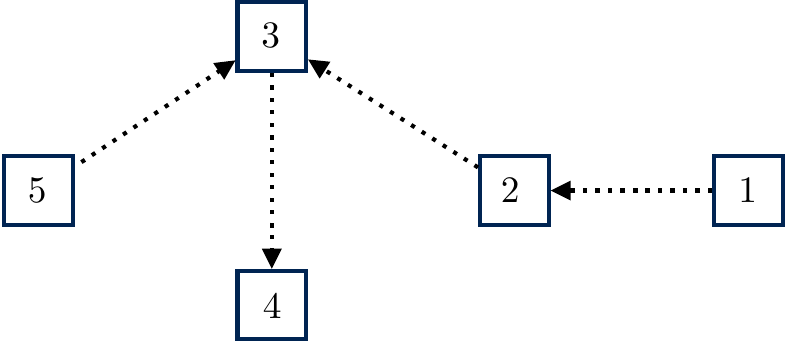}
	\caption{\textbf{Working example}: \emph{Local information exchange}. Dotted line arrows represent the established communication links between agents. }
	\label{fig::InfStr_DS}
\end{figure}

In contrast to the previous section where agent  $ j\in\mc{N}_i $ communicates to agent $i$ explicitly the functional form of its  states, in the proposed framework agent  $ j$  communicates a compact set $ \mc X_j $, henceforth referred as the \emph{state forecast set}, that contains possible evolution of its states, i.e., $\bs x_j \in \mc X_j$. In this framework, the shape of $ \mc X_j $ is a decision quantity for each agent $ j $. Upon receiving these state forecast sets from all of its neighbors, agent $ i $ treats these states as uncertain quantities that affect its dynamics in a similar way as exogenous disturbances. To emphasize this, we denote by $ \zeta_{j,t}\in\mathbb{R}^{n_{x,j}} $ the belief of agent $ i $ about the states of agent $j$ at time $t$, such that $\bs\zeta_j=[\zeta_{j,t}]_{t\in\mc{T}_+} \in \mc{X}_j$. In this context, the policy of agent $i$ at time $t$ is based on the information from its own states $x_{i,t}$ and on the belief states $ \bs \zeta_{\Ni,t} $, from stage 1 up to $t$. This leads us to design \emph{causal local state/disturbance feedback policies} $ \psi_{i,t}:\mb R^{\hat n_{x,i}^t} \rightarrow \mb R^{n_u^i} $ where $ \hat {n}_{x,i}^t = t\left(n_{x,i} + \sum_{j\in \mc N_i} n_{x,j}\right) $, such that the  input of agent $i$ at time $ t $ is given by $ u_{i,t} = \psi_{i,t}(\bs x_i^t, \bs \zeta_\Ni^t) $. We denote by $ \bs {\psi}_{i}(\bs x_i, \bs \zeta_{\Ni}) = [\psi_{i,t}(\bs x_i^t, \bs \zeta_\Ni^t)]_{t \in \mc T} $ the policy concatenation over the time horizon, and by $ \mc C(\bs x_i, \bs \zeta_{\Ni}) $ the corresponding space of causal local state/disturbance feedback policies associate with agent $i$.

In this decentralized setting, the robust optimization problem to design policies with local information exchange which minimize the sum of worst-case individual objectives is formulated as
\begin{equation}\label{Decentralized}
\begin{array}{l}
\text{\;\;minimize } \displaystyle\sum\limits_{i = 1}^M \max\limits_{\bs \xi_i \in \Xi_i, \bs \zeta_{\Ni} \in \mc X_{\Ni}} J_i(\bs x_i,\bs u_i) \\
\left.\begin{array}{@{}r@{\;}l@{}}
\text{subject to }& \bs {\psi}_{i}(\cdot) \in \mc C(\bs x_i, \bs \zeta_{\Ni}),\, \bs u_i = \bs {\psi}_{i}(\bs x_i, \bs \zeta_{\Ni})\\
& \bs x_{i} = f_{i}\big(\bs \zeta_{\Ni}, \bs u_{i}, \bs \xi_{i}\big)\\
& (\bs x_{i}, \bs u_{i}) \in \mc O_i \\
& \bs x_i \in \mc X_i, \; \mc X_i \in \mc{F}(\mb R^{N_{x,i}})
\end{array}\right \rbrace \begin{array}{@{}l}
\forall \bs \zeta_{\Ni} \in \mc X_{\Ni}\\
\forall \bs \xi_i \in \Xi_i
\end{array}\;\forall i \in \mc M,
\end{array}
\tag{$\text{L}:\mc C(\bs x_i, \bs \zeta_{\Ni}), \mc{F}(\mb R^{N_{x,i}})$}
\end{equation}
where $ \mc X_{\Ni} = \bigtimes_{j \in \mc N_i} \mc X_j $, and $\mc F(\mb R^{N_{x,i}}) $, where $ N_{x,i} = (T+1)n_{x,i} $, denotes the field of sets generated by all the compact subsets of the power set of $ \mb R^{N_{x,i}} $. The optimization variables  are $ \bs {\psi}_{i}(\cdot) $ and $ \mc X_i $ for all $ i \in \mc M $.  Since agent $i$ treads the beliefs $\bs \zeta_{\Ni}$ as exogenous disturbances,  its decisions are taken in view of the worst-case both with respect to  $\bs \xi_i \in \Xi_i$  and $\bs \zeta_{\Ni} \in \mc X_{\Ni}$. This is also reflected in the construction of  the objective function. Notice that agent $i$ is not directly affected by the disturbances $\bs \xi_j \in \Xi_j$ of its neighbors. Rather, the effect of $\bs \xi_j \in \Xi_j$ of agent $j\in\mc N_i$ is been translated into  set $\mc X_j$, which in turn affects agent $i$.

\PDs can be interpreted as a method for finding a compromise between the agents as this is represented through sets $\mc X_i$. If we focus attention in two agents, agents $i$ and $j$ with $j\in \mc{N}_i$, on the one hand  agent $j$ will benefit the most if the set $\mc X_j$ is as large as possible. By doing that, set  $\mc X_j$  will not impose any additional constraints on its states, thus agent $j$ will individually achieve   the lowest objective value contribution. On the other  hand, agent $i$ will benefit the most if it receives from agent $j$ the smallest possible set $\mc X_j$ (preferable $\mc X_j$ is a  singleton) which will reduce the uncertainty on its dynamics, and will have a positive effect in achieving the lowest objective value contribution. Since the objective of \PDs is to minimize the equally weighted sum of individual worst-case costs, the resulting policy/set pair finds the trade-off among the agents and achieving the lowest, network wide, objective value, while achieving robustly  feasibility with respect to  $\bs{\xi}_\Mu \in \Xi_\Mu$.

\PDs addresses the privacy concerns in the following ways. First,  the local communication network ensures that the information exchange is the minimum among the agents. Second, focusing again on agents $i$ and $j$ with $j\in \mc{N}_i$,  agent $j$ does  not directly reveal the function form of its states to agent $i$, which is the case in both the centralized and partially nested information exchange. Rather,   the future state trajectories are ``masks" by  set $\mc X_j$, thus reducing exposure to  agent $i$ who might want to leverage on the knowledge gained about the  constraints and dynamics of agent $j$.  If additionally the set $\mc X_j$ has a simple structure, e.g., $\mc X_j$ is rectangular, agent $j$ will reveal  the bare minimum to achieve a compromise in cost in terms objective of \PDs.  This will be studied further in the next section. 
Notice that \PDs has  highly decoupled structure, with only sets $\mc X_j$ linking the agents in the system.

The state feedback policies $\mc C(\bs x_i, \bs \zeta_{\Ni})$ induces a non-convex  optimization problem. As in \cite{HGK:2011b,Lin2016},  we now  focus on purely disturbance feedback policies $ \Psi_{i,t}:\mb R^{\hat n_{\xi,i}^t} \rightarrow \mb R^{n_{u,i}} $ where $ \hat n_{\xi,i}^t = (t-1)n_{\xi,i} + t\sum_{j\in \mc N_i} n_{x,j} $, such that the input at time $t$ is $ u_{i,t} = \Psi_{i,t}(\bs \xi_i^{t-1},\bs \zeta_\Ni^{t}) $. We write $ \bs \Psi_i(\bs \xi_i, \bs \zeta_{\Ni})  = [\Psi_{i,t}(\bs \xi_i^{t-1},\bs \zeta_\Ni^{t}]_{t\in\mc T} $ to denote the policy concatenation over time, and we define as $ \mc {SC}(\bs \xi_i, \bs \zeta_{\Ni}) $ the space of strictly causal disturbance feedback policies. Note that  the ``strictly causal" refers to the uncertain vector $\bs \xi_{i}^{t-1}$ which the policy is allowed to depend up to stage $t-1$, while the policy is ``causal" in state beliefs $ \bs \zeta_\Ni^{t} $ and is allowed to depend up to stage $t$.  The following theorem, establishes the equivalence between the proposed state/disturbance and disturbance feedback policies.
\begin{theorem}\label{thm::2}
	\PDs is a equivalent to \PDd in the following sense: Given a feasible state/disturbance feedback policy $ \bs \psi_{i}(\cdot) $ for \PDs, a feasible disturbance feedback policy $ \bs \Psi_{i}(\cdot) $ for \PDd can be constructed that achieves the same objective value, and vice versa. 
\end{theorem}

The key difference between the partially nested information \PSCs and the local information \PDs, is that the synthesis phase of the latter requires that each agent communicates only with its direct neighbors rather than with all its precedent agents in the network (compare Fig. \ref{fig::InfStr_NS} to Fig. \ref{fig::InfStr_DS} for the working example). This minimum communication exchange is sufficient to address \PDs since the coupling among agents in the network only appears through sets $ \mc X_i $. This however introduces a level of conservativeness which is formalized in the following theorem. 

\begin{theorem}\label{thm::3}
	\PDd is a conservative approximation of \PSCd in the following sense: every feasible solution of \PDd is feasible in \PSCd, and the optimal value of \PDd is equal or larger than the optimal value of \PSCd.
\end{theorem}

The following corollary summarizes the relation, in terms of optimal value and cost, between the centralized \PCd and local information exchange \PDd, which is an immediate implication of Theorems \ref{thm::1} and \ref{thm::3}.
\begin{corollary}\label{corr::1}
	\PDd is a conservative approximation of \PCd in the following sense: every feasible solution of \PDd is feasible in \PCd, and the optimal value of \PDd is equal or larger than the optimal value of \PCd.
\end{corollary}

\PDs is computationally intractable because $ (i) $ the optimization of the policies is performed over the infinite space of causal functions; $ (ii) $ the optimization of the state forecast sets $ \mc X_i $ is performed over arbitrary sets; and $ (iii) $ the constraints must be satisfied robustly for every uncertain realization.

\section{Solution method} \label{sec::SolMethod}

In this section, we discuss appropriate restrictions to \PDd that allow us to obtain a computationally tractable approximation. We begin by restricting each state forecast set $ \mc X_i $, with $ i \in \mc M $, to admit a convex conic structure. We denote by $ \mc F_{\CC}(\mb R^{N_{x,i}}) $ the field of sets generated by all the convex conic compact subsets of the power set of $ \mb R^{N_{x,i}} $. This design choice is motivated by the fact that every convex set admits a conic representation \cite[p. 15]{Rockafellar2015}. \PDdCC still remains computational intractable. As shown in Theorem \eqref{thm::4}, this is indeed the case even when its policies admit an affine structure, (e.g., see \cite{DanielPrimalDual,bental2004ars}), i.e., $ u_{i,t} = \Psi_{i,t}(\bs \xi_i^{t-1},\bs \zeta_\Ni^{t}) $ with
\begin{equation*}
\Psi_{i,t}(\bs \xi_i^{t-1},\bs \zeta_\Ni^{t}) =  v_{i,t} +   V_{i,t} \bs\xi_i^{t-1} + V_{\Ni,t} \bs \zeta_\Ni^{t}  ,
\end{equation*}
where $ v_{i,t} $,  $ V_{i,t} $ and  $ V_{\Ni,t} $ are the decision variables matrices, with appropriate dimensions, that define the affine policy. We denote by $\mc {SC}_{a}(\bs \xi_i, \bs \zeta_{\mc N_i})$ the finite dimensional space of affine policies, and refer to \PDd restricted to affine decision policies and convex conic state forecast sets as \PDdCCa. 
\begin{theorem} \label{thm::4}
	\PDdCCa is NP-hard.
\end{theorem}

To gain tractability, we further restrict  $ \mc X_i$ to  sets that can be represented through an affine transformation of a fixed set, in a similar spirit as \cite{Jaillet2016,Zhang2017,Bitlislioglu2017}. Considering agent $i$, for a given $K_i\in\mathbb{Z}_+$ we define matrices $P_{i,k}\in\mathbb{R}^{N_{x,i} \times N_{x,i}} $ to be orthogonal projection of the state $\bs{x}_i$ such that 
\begin{equation*}
\bs{x}_i = \sum_{k=1}^{K_i} P_{i,k}\bs{x}_{i}.
\end{equation*}
For a given convex conic compact set  $ \mc S_i $,  $ \mc X_i$ is restricted to 
\begin{equation}\label{eq::SFSv1}
\begin{array}{l}
\mc X_i(y_i,z_i) = \displaystyle\left\{\bs{x}_i \in\mathbb{R}^{N_{x,i}} \,:\, \exists \bs s_{i} \in \mc S_i \text{ s.t. } \bs x_{i} = \sum_{k=1}^{K_i} y_{i,k} P_{i,k} \bs s_{i} + z_{i} \right\}
\end{array}
\end{equation}
where the decision variables that define the shape of the state forecast set are vectors $y_i = [y_{i,1},\ldots,y_{i,K_i}]^\top\in\mathbb{R}^{K_i}_+$ and $ z_{i}\in\mathbb{R}^{N_{x,i}}$. Set $\mc S_i$ is expressed using given matrices $ G_{i,k} \in \mathbb{R}^{\ell_{k,i} \times N_{x,i}} $, vectors $ g_{i,k} \in \mathbb{R}^{\ell_k }$ and convex cones $ \mc K_{i,k} $, as follows
\begin{equation}
\begin{array}{l}
\mc S_i = \Big\{\bs{s}_{i} \in \mb R^{N_{x,i}}\;:\;  G_{i,k} P_{i,k} \bs s_{i} \preceq_{\mc K_{i,k}}  g_{i,k},\; k=1,\ldots,K_i \Big\}.
\end{array}
\end{equation}
The positive scalar decision $y_{i,k}$ allows to scale the $k$-th constraint  $G_{i,k} P_{i,k} \bs s_{i} \preceq_{\mc K_{i,k}} g_{i,k}$ thus controlling the shape of the state forecast set in the direction defined by the states in the projection $P_{i,k}\bs{x}_i$, while vector $z_i$ is responsible for the translation of the set. Henceforth, we denote by $ \mc F_{\AT}(\mc S_i) $ the field of bounded convex conic sets, $ \wh{\mc X}_i(y_i, z_i) $, that can be represented through an affine transformation of a fixed set $ \mc S_i $.

\begin{example}
	Consider the state of agent $i$ such that $\bs x_{i}\in\R^2$. If we want to enclose states $ x_{i,1}$ and $ x_{i,2}$ within a hyper-rectangular, we first set   $P_{i,1} = \left(\begin{smallmatrix}1&0\\0&0\end{smallmatrix}\right)$ and $P_{i,2} = \left(\begin{smallmatrix}0&0\\0&1\end{smallmatrix}\right)$, and $\mc S_i$ is given by 
	\begin{equation*}
	\mc S_i = \Big\{\bs{s}_{i} \in \mb R^{2}\;:\; \|{s}_{i,1}\|_\infty \leq 1 ,\,\|s_{i,2}\|_\infty \leq 1 \Big\}
	\end{equation*}
	implying that 	$ G_{i,1} = \left(\begin{smallmatrix}0&0\\-1&0\end{smallmatrix}\right)$, $ g_{i,1} = \left(\begin{smallmatrix}1\\0\end{smallmatrix}\right)$ and $ G_{i,2} = \left(\begin{smallmatrix}0&0\\0&-1\end{smallmatrix}\right)$, $ g_{i,2} = \left(\begin{smallmatrix}1\\0\end{smallmatrix}\right)$, and $\mc K_{i,1}$, $\mc K_{i,2}$ are both infinity cones. Such a construction is also graphically illustrated in Fig.~\ref{fig::PrimitiveSets}.
	\begin{figure}[ht]
		\centering
		\includegraphics[width = 0.6\textwidth]{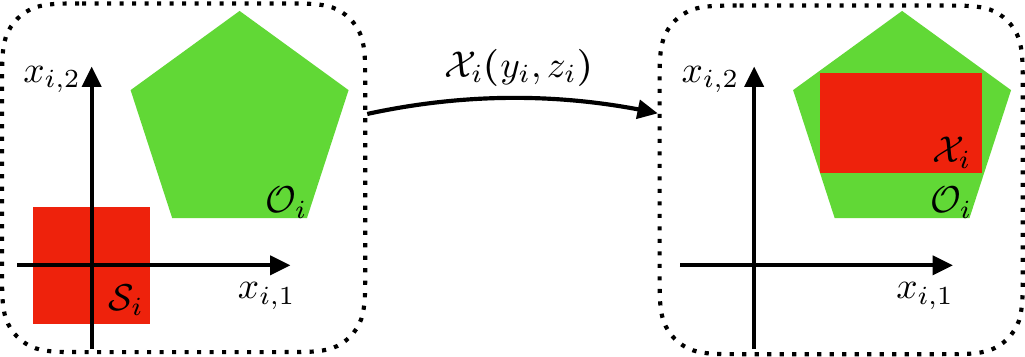}
		\caption{Example in $ \mb R^2 $ of $ \mc X_i $ construction using a rectangular $ \mc S_i $.}
		\label{fig::PrimitiveSets}
	\end{figure}
\end{example}

\begin{example}  
	Consider the state of agent $i$ such that $\bs x_{i}\in\R^3$. If we want to enclose states $ x_{i,1}$ and $ x_{i,2}$ within a two dimensional circle, and state $ x_{i,3}$ within a one dimensional box, we first set   $P_{i,1} = \left(\begin{smallmatrix}1&0&0\\0&1&0\\0&0&0\end{smallmatrix}\right)$ and $P_{i,2} = \left(\begin{smallmatrix}0&0&0\\0&0&0\\0&0&1\end{smallmatrix}\right)$, and $\mc S_i$ is given by 
	\begin{equation*}
	\mc S_i = \Big\{\bs{s}_{i} \in \mb R^{3}\;:\; \|({s}_{i,1},{s}_{i,2})\|_2 \leq 1 ,\,\| s_{i,3}\|_2\leq 1 \Big\}
	\end{equation*}
	implying that $ G_{i,1} = \left(\begin{smallmatrix}0&0&0\\-1&0&0\\0&-1&0\end{smallmatrix}\right)$, $ g_{i,1} = \left(\begin{smallmatrix}1\\0\\0\end{smallmatrix}\right)$ and $ G_{i,2} = \left(\begin{smallmatrix}0&0&0\\0&0&-1\end{smallmatrix}\right)$, $ g_{i,2} = \left(\begin{smallmatrix}1\\0\end{smallmatrix}\right)$, and $\mc K_{i,1}$, $\mc K_{i,2}$ are both second order cones.
\end{example}

Approximation \eqref{eq::SFSv1} is more restrictive than the one proposed in \cite{Jaillet2016,Zhang2017,Bitlislioglu2017} which allows for arbitrary translation/rotation of set $\mathcal{S}_i$. However, this additional restriction is crucial since set \eqref{eq::SFSv1} is in generally  a non-convex region  due to the multiplication between $ y_{i,k}$ and $ \bs s_{i} $ which are both decision variables in the resulting optimization problem. By taking advantage of the fact that $ \mc S_i $ is compact, $ \mc X_i(y_i,z_i) $ can be expressed as the following convex set. 
\begin{equation}\label{eq::SFSv2}
\begin{array}{l@{\,}l}
\mc {\widehat X}_i(y_i,z_i) = \displaystyle\Bigg\{\bs{x}_i \,:\, \exists \bs  \nu_{i,k}\in\mathbb{R}^{N_{x,i}} \text{ s.t. } & \displaystyle\bs x_{i} = \sum_{k=1}^{K_i} P_{i,k} \bs \nu_{i,k} + z_{i},\\
& G_{i,k} P_{i,k} \bs \nu_{i,k}  \preceq_{\mc K_{i,k}}  y_{i,k}  g_{i,k},\;k=1,\ldots,K_i \Bigg\}
\end{array}
\end{equation}
where $ \bs \nu_{i,k} $ are auxiliary variables. The relationship between sets \eqref{eq::SFSv1} and \eqref{eq::SFSv2} is summarized in the following proposition. 
\begin{proposition}\label{prop::nCc}
	Set $ \mc X_{\FS} = \{(\bs x_i, y_i, z_i)\,:\, \bs x_i \in \mc X_i(y_i, z_i) \} $ is equivalent to $ \wh{\mc X}_{\FS} =\{(\bs x_i, y_i, z_i)\,:\, \bs x_i \in \mc {\widehat X}_i(y_i,z_i) \} $ in the following sense: there exist a unique mapping between feasible points in $ {\mc X}_{\FS} $ and $ \wh{\mc X}_{\FS} $.
\end{proposition}

The restriction of \PDdCC to state forecast sets $ \wh{\mc X}_i(y_i, z_i) $ is given as
\begin{equation}\label{DecentralizedXab}
\begin{array}{@{}l}
\text{\;\;minimize\,} \displaystyle\sum\limits_{i = 1}^M \max\limits_{\bs \xi_i \in \Xi_i, \bs \zeta_{\Ni} \in \wh{\mc X}_{\Ni}} J_i(\bs x_i,\bs u_i) \\
\left.\begin{array}{@{}r@{\;}l@{}}
\text{subject to\,}& \bs {\psi}_{i}(\cdot) \in \mc {SC}(\bs \xi_i, \bs \zeta_{\Ni}), \bs u_i = \bs {\Psi}_{i}(\bs \xi_i, \bs \zeta_{\Ni})\\
& \bs x_{i} = f_{i}\big(\bs \zeta_{\Ni}, \bs u_{i}, \bs \xi_{i}\big)\\
& (\bs x_{i}, \bs u_{i}) \in \mc O_i \\
& \bs x_i \in \wh{\mc X}_i(y_i, z_i),\, \wh{\mc X}_i(\cdot) \in \mc F_{\AT}(\mc S_i)
\end{array}\right \rbrace \begin{array}{@{}l}
\forall \bs \zeta_{\Ni} \in \wh{\mc X}_{\Ni}(y_\Ni, z_\Ni)\\
\forall \bs \xi_i \in \Xi_i
\end{array}\forall i \in \mc M,
\end{array}
\tag{$\text{L}:\mc {SC}(\bs \xi_i, \bs \zeta_{\Ni}), \mc{F}_\AT(\mc S_i)$}
\end{equation}
where for each $i\in \mc M$ we define $ \wh{\mc X}_{\Ni}(y_{\Ni}, z_\Ni) = \bigtimes_{j \in \mc N_i} \wh{\mc X}_j(y_{j}, z_j) $. The optimization variables are $ \bs {\Psi}_{i}(\cdot) $, $ y_i $ and $ z_i $ for all $ i \in \mc M $. \PDdAT has a semi-infinite structure with decision-dependent uncertainty sets. Thus, in general, it admits a non-convex feasible region because of the decision-dependent uncertainty sets. This can be verified by noticing that reformulating any semi-infinite robust constraint in the problem leads to a non-convex set of inequality and equality constraints. To obtain a problem formulated with decision-independent uncertainty sets, we propose, in the spirit of \cite{Jaillet2016,Zhang2017,Bitlislioglu2017}, the design of affine causal feedback policies $ \Gamma_{i,t}:\mb R^{\bar n_{i}^t} \rightarrow \mb R^{n_{u,i}} $ such that the control input at each time step is given by $ u_{i,t} = \Gamma_{i,t}(\bs \xi_i^{t-1}, \bs s^t_\Ni) $. We write $ \bs \Gamma_i(\bs \xi_i, \bs s_{\Ni})  = [\Gamma_{i,t}(\bs \xi_i^{t-1}, \bs s^t_\Ni)]_{t\in\mc T} $ to denote the policy concatenation over time, and we define as $ \mc {SC}(\bs \xi_i, \bs s_\Ni) $ the space of strictly causal disturbance feedback policies. In this context, the counterpart of \PDdAT is given as
\begin{equation}\label{DecentralizedFinal}
\begin{array}{@{}l}
\text{\;\;minimize } \displaystyle\sum\limits_{i = 1}^M \max\limits_{\bs \xi_i \in \Xi_i, \bs s_{\Ni} \in \mc S_{\Ni}} J_i(\bs x_i,\bs u_i) \\
\left.\begin{array}{@{}r@{\;}l@{}}
\text{subject to }& \bs {\Gamma}_{i}(\cdot) \in \mc {SC}(\bs \xi_i, \bs s_{\Ni}), \,\bs u_i = \bs {\Gamma}_{i}(\bs \xi_i, \bs s_{\Ni})\\
& \bs \zeta_{\Ni} = Y_\Ni \bs s_\Ni + z_\Ni\\ 
& \bs x_{i} = f_{i}\big(\bs \zeta_{\Ni}, \bs u_{i}, \bs \xi_{i}\big)\\
& (\bs x_{i}, \bs u_{i}) \in \mc O_i \\
& \bs x_i \in \wh{\mc X}_i(y_i, z_i),\, \wh{\mc X}_i(\cdot) \in \mc F_{\AT}(\mc S_i)
\end{array}\right \rbrace \begin{array}{@{}l}
\forall \bs s_{\Ni} \in \mc S_{\Ni}\\
\forall \bs \xi_i \in \Xi_i
\end{array}\forall i \in \mc M,
\end{array}
\tag{$\text{L}:\mc{SC}(\bs \xi_i, \bs s_{\Ni}), \mc{F}_\AT(\mc S_i)$}
\end{equation}
where $ \mc S_{\Ni} = \bigtimes_{j \in \mc N_i} \mc S_j $ and, with a slight abuse of notation, $ Y_i = \sum_{k=1}^{K_i} y_{i,k} P_{i,k} $ and $ Y_\Ni \bs s_\Ni + z_\Ni = \left[ Y_i \bs s_{j} + z_j \right]_{j \in \mc N_i} $ for each $ i \in \mc M $. The decision variables are $ \bs \Gamma_i(\cdot) $, $ y_i $ and $ z_i $ for all $ i \in \mc M $. Note that for \PDdATC to be computationally tractable, we further need to restrict the infinite-dimensional structure of its decision policies, e.g., to admit an affine structure, as suggested in \cite{DanielPrimalDual,bental2004ars}. 

In the following, we will show the relationship between \PDdATC and \PDdAT. To this end, we define the linear mapping $ L_{i,t}: \mb R^{n_{x,i}} \rightarrow \mb R^{n_{x,i}} $, $ s_{i,t} \mapsto x_{i,t} $, as,
\begin{equation}\label{app::map1}
L_{i,t}(s_{i,t}) = Y_{i,t} s_{i,t} + z_{i,t},
\end{equation}
and the linear mapping, $ R_{i,t}: \mb R^{n_{x,i}} \rightarrow \mb R^{n_{x,i}} $, $ x_{i,t} \mapsto s_{i,t} $, as,
\begin{equation}\label{app::map2}
R_{i,t}(x_{i,t}) = Y_{i,t}^{+}(x_{i,t} -z_{i,t})
\end{equation}
where $  Y^+_{i,t} := (Y_{i,t}^\top Y_{i,t})^{-1}Y^\top_{i,t} $ is the pseudo-inverse of the positive semi-definite matrix $ Y_{i,t} $. Note that the mapping $ R_{i,t} $ may not be unique because of the pseudo-inverse $ Y^+_{i,t} $. Moreover, $ L_{i,t} $ can be viewed as a ``left inverse'' of the operator $ R_{i,t} $, i.e., it satisfies $ L_{i,t}\big(R_{i,t}(x_{i,t})\big) = x_{i,t}  $. Using this mapping, the following theorem establishes equivalence between \PDdATC and \PDdAT.

\begin{theorem} \label{thm::5}
	\PDdATC is equivalent to \PDdAT in the following sense: there exist a mapping between feasible solutions in \PDdATC and \PDdAT. Moreover, the optimal value of \PDdATC is equal to the optimal value of \PDdAT.
\end{theorem}

In view of Theorem \eqref{thm::5}, the mapping between feasible solutions in \PDdATC and \PDdAT can now be defined. This allows a local controller to evaluate the resulting policy based on the established local communication network. Given $ u_{i,t} = \Gamma_{i,t}(\bs \xi_i^{t-1}, \bs s_\Ni^t) $ as the optimal control policy derived from the solution of \PDdATC, then $ u_{i,t} = \Psi_{i,t}(\bs \xi_i^{t-1}, \bs \zeta_\Ni^t) = \Gamma_{i,t}(\bs \xi_i^{t-1},[ R_j^t(\bs \zeta_j^t)]_{j\in \mc N_i}) $ is the optimal control policy for \PDdAT.

\PDdATC retains the coupling structure of \PDs since agent $ i $ only needs to communicate to its direct neighbors the scaling $ y_i $ and translation $ z_i $ of its predefined fixed set $ \mc S_i $. Contrary to \PCd and \PSCd, there is no coupling introduced among the agents due to the form of their decision policies. This is because the agents treat the effect of their neighbors as a disturbance to their systems. In this framework, we can claim that the proposed method alleviates concerns on privacy since this limited communication scheme does not reveal the exact characteristics of the dynamics within each system. Moreover, the loosely coupled structure of \PDdATC makes it amendable to distributed computation algorithms as ADMM to efficiently solve it \cite{Boyd2004}.

\section{Numerical results} \label{sec::Numerics}
In this section, we conduct a number of simulation-based studies to assess the efficacy of the proposed decentralized controller synthesis approach. We focus our attention on three examples: $ (i) $ A toy example that allow us to illustrate, numerically and graphically, the connection between the shape of the primitive sets and the solution quality; $ (ii) $ A system composed of masses that are connected by springs and dampers which is suitable to study the scalability and the closed-loop behavior of the proposed methodology; $ (iii) $ a supply chain operated in a distributed decision making authority where we exhibit the efficacy of the proposed method as a contract design mechanism and we investigate its performance on various coupling network structures.

\subsection{Example 1: Illustrative example}

\begin{figure}[ht]
	\includegraphics[width=0.6\textwidth]{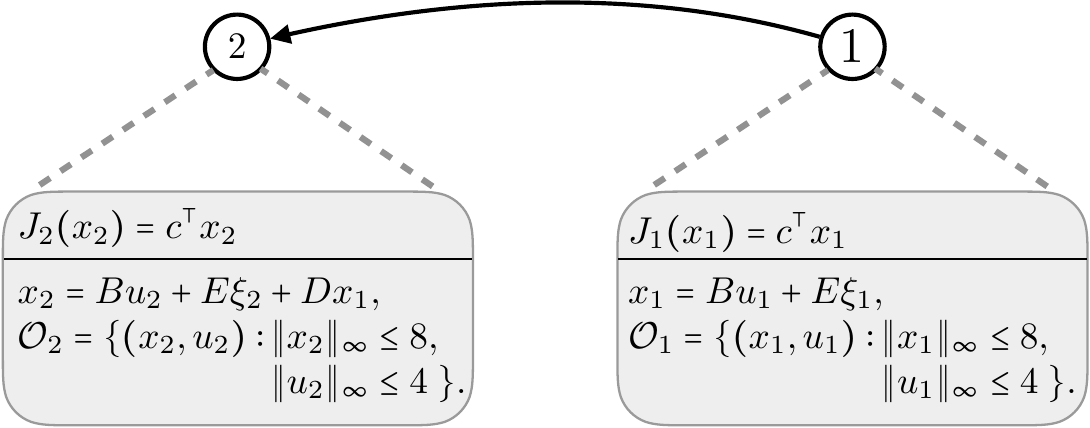}
	\caption{Physical and information structure of the two agents in the system.}
	\label{fig::toyExample}
\end{figure}

We consider a system composed by two agents with states $ x_1, x_2 \in \mb R^2 $, inputs $ u_1, u_2 \in \mb R $ and disturbances $ \xi_1, \xi_2 \in \Xi = \{\xi \in \mb R^2: \| \xi \|_\infty \leq 1 \} $, respectively. The nested physical and information communication network along with the objective functions and constraint sets of the agents are shown in Fig. \ref{fig::toyExample}. The system matrices are given as,
\begin{equation}
	c = \begin{bmatrix}
	1 \\ -1
	\end{bmatrix}, \;B = \begin{bmatrix}
	1 \\ 0.8
	\end{bmatrix},\; E = \begin{bmatrix}
	1 & -1 \\ -1 & 1
	\end{bmatrix} \text{ and } D = \begin{bmatrix}
	1 & 0 \\ 0 & -2
	\end{bmatrix}.
\end{equation}

The system-wide robust optimization problem is formulated as follows:
\begin{equation}
\label{problem_toyExample_Cent}
	\begin{array}{@{}l@{}}
		\min \max\limits_{\xi_1, \xi_2} J_1(x_1) + \max\limits_{\xi_1, \xi_2} J_2(x_2) \\
		\left.\begin{array}{r@{\;}l@{}}
		\st& u_1 \in \mc {SC}_1, u_2 \in \mc {SC}_2,\\
		& x_1 = B u_1 + E \xi_1,\\
		& (x_1,u_1) \in \mc O_1,\\
		& x_2 = B u_2 + E \xi_2 + D x_1,\\
		& (x_1,u_1) \in \mc O_2,
		\end{array}\right \rbrace \begin{array}{@{}l}
		\forall \xi_1, \xi_2 \in \Xi.
		\end{array}
		\end{array}
\end{equation}
with $ \mc {SC}_1, \mc {SC}_2 \subseteq \mc {SC}(\xi_1,\xi_2) $ where $ \mc {SC}(\cdot) $ denotes the infinite dimensional function spaces generated by strictly causal disturbance feedback policies on its arguments. In this setting, the equivalent to \PCd for the synthesis of optimal centralized controllers is obtained by restricting $ \mc {SC}_1 = \mc {SC}_2 = \mc {SC}(\xi_1,\xi_2) $. Similarly, the equivalent to \PSCd for the synthesis of optimal nested-information controllers is obtained by restricting $ \mc {SC}_1 = \mc {SC}(\xi_1) $ and $ \mc {SC}_2 = \mc {SC}(\xi_1,\xi_2) $. Both problems are solved by resorting to affine disturbance feedback policies. 

\begin{figure}[h]
	\begin{minipage}{0.33\textwidth}
		\subfigure[]{\includegraphics[width = \textwidth]{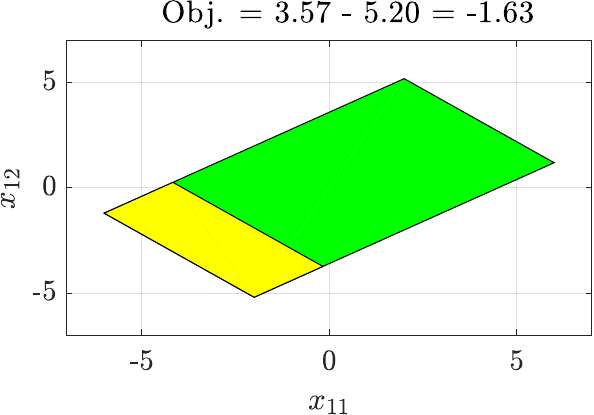}\label{fig::centralizedSynthesis_a}}
	\end{minipage}~
	\begin{minipage}{0.33\textwidth}
		\subfigure[]{\includegraphics[width = \textwidth]{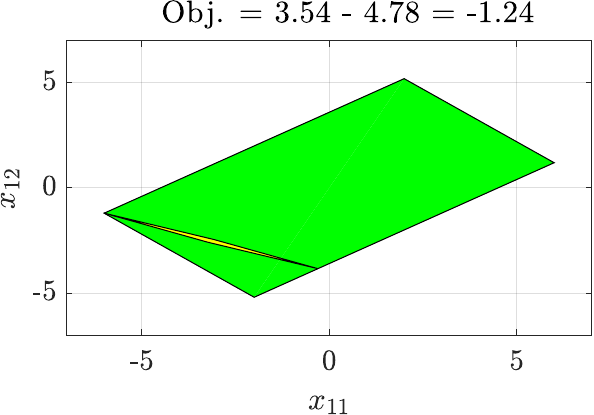}\label{fig::centralizedSynthesis_b}}
	\end{minipage}
	\caption{(green) Feasible set for $ x_1(w) $, (yellow) $ x_1(w) $ for Problem \eqref{problem_toyExample_Cent} with (a)  $ \mc {SC}_1 = \mc {SC}_2 = \mc {SC}(\xi_1,\xi_2) $ and (b) $ \mc {SC}_1 = \mc {SC}(\xi_1) $ and $ \mc {SC}_2 = \mc {SC}(\xi_1,\xi_2) $.}
	\label{fig::centralizedSynthesis}
\end{figure}
In Fig. \ref{fig::centralizedSynthesis} we depict in green the feasible set for $ x_1 $ of Problem \eqref{problem_toyExample_Cent} when $ \mc {SC}_1 = \mc {SC}_2 = \mc {SC}(\xi_1, \xi_2) $. In addition to this, Fig. \ref{fig::centralizedSynthesis_a} shows in yellow the region generated by the optimal centralized policy for $ x_1 $, while Fig. \ref{fig::centralizedSynthesis_b} shows in yellow the respective region for the optimal decentralized policy of it. To streamline the presentation, henceforth we refer to yellow regions as optimal centralized and nested-information, respectively. The resulting objective values are reported in the title of the respective figure as ``obj. = $ J_1(x_1) + J_2(x_2) $''. We observe that the information restriction imposed on the nested-information controllers synthesis results in a larger objective value and a smaller optimal region with respect to the centralized solution.

The synthesis of robust decentralized controllers using the communicated sets approach is given as,
\begin{equation}
\label{problem_toyExample_Decent}
\begin{array}{@{}l@{}}
\min \max\limits_{\xi_1} J_1(x_1) + \max\limits_{s_1, \xi_2} J_2(x_2) \\
\left.\begin{array}{r@{\;}l@{}}
\st& u_1 \in \mc C(\xi_1), u_2 \in \mc C(\xi_2,\zeta_1),\\
& x_1 = B u_1 + E \xi_1,\\
& (x_1,u_1) \in \mc O_1,\\
& x_1 \in \mc X_1 = y_1 \mc S_1 + z_i,\\
& \zeta_1 = y_1 s_1 + z_1,\\
& x_2 = B u_2 + E \xi_2 + D \zeta_1,\\
& (x_1,u_1) \in \mc O_2,
\end{array}\right \rbrace \begin{array}{@{}l}
\forall s_1 \in \mc S_1,\\
\forall \xi_1, \xi_2 \in \Xi.
\end{array}
\end{array}
\end{equation}
This problem is an instance of \PDdAT where the state forecast set $ \mc X_1 $ is expressed as the scaling, $ y_1 $, and translation, $ z_1 $, of a predefined primitive sets, $ \mc S_1 $. To solve it we use affine feedback policies to restrict the infinite structure of its decision variables.

\begin{figure}[h]
	\begin{minipage}{0.33\textwidth}
		\subfigure[]{\includegraphics[width = \textwidth]{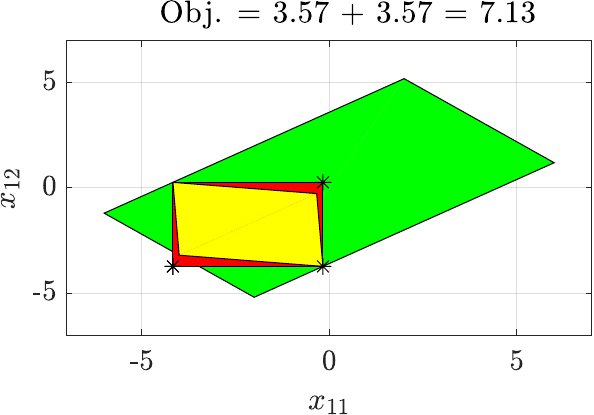}}
	\end{minipage}\hfill
	\begin{minipage}{0.33\textwidth}
		\subfigure[]{\includegraphics[width = \textwidth]{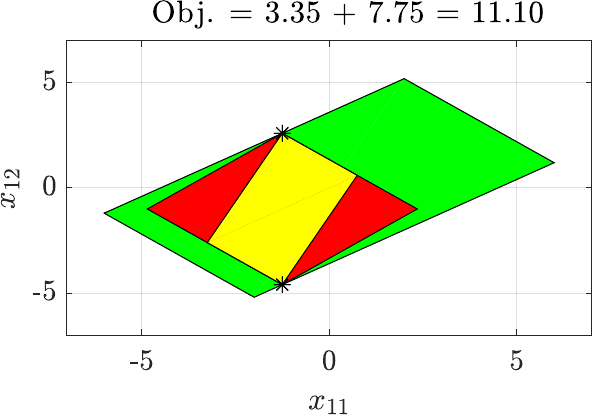}}
	\end{minipage}\hfill
	\begin{minipage}{0.33\textwidth}
		\subfigure[]{\includegraphics[width = \textwidth]{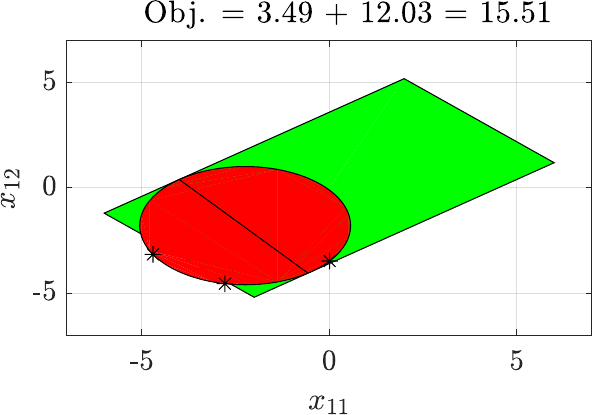}}
	\end{minipage}
	\caption{(green) Feasible set for $ x_1(w) $, (yellow) $ x_1(w) $ for Problem \eqref{problem_toyExample_Decent} with $ \mc S_1 $ as \mbox{(a) box}, \mbox{(b) rhombus}, and \mbox{(c) circle}.}
	\label{fig::decentralizedNorms}
\end{figure}
In what follows we investigate how the quality of the solution, in terms of objective value, is affected by the choice of the primitive set. In Fig. \eqref{fig::decentralizedNorms}, this comparison is performed with respect to a box, rhombus and circle primitive set. As previously mentioned, the area in green depicts the feasible set for $ x_1 $, the area in yellow depicts the region generated by the optimal policy of $ x_1 $ after solving Problem \eqref{problem_toyExample_Decent} using the respective primitive set. In addition, we show in red the area covered by the set $ \mc X_1 $ communicated by agent 1 to agent 2 and with black stars the worst-case scenarios for agent 2 in the state forecast set $ \mc X_1 $. We observe that the optimal region of agent 1 changes with the different primitive sets as an attempt to cooperate with agent 2. This cooperative behavior is also identified in the objective values as agent 1 ``sacrifices'' some of its optimality for the good of agent 2. Interestingly enough, part of the state forecast set $ \mc X_1 $ may lie outside the feasible region of the problem which adds conservativeness to the system as can be verified by inspecting the position of the worst-case scenarios. Moreover, since the worst-case scenarios are not necessarily placed to the corner points defined by the optimal region, agent 1 retains some of its privacy since the behavior of its optimal policy is not revealed to agent 2.

\begin{figure}[h]
	\begin{minipage}{0.33\textwidth}
		\subfigure[]{\includegraphics[width = \textwidth]{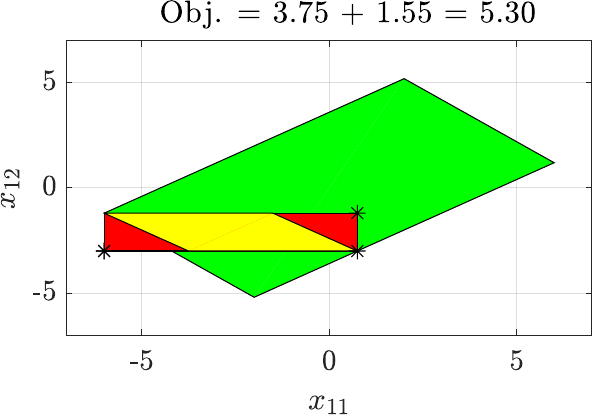}}
	\end{minipage}\hfill
	\begin{minipage}{0.33\textwidth}
		\subfigure[]{\includegraphics[width = \textwidth]{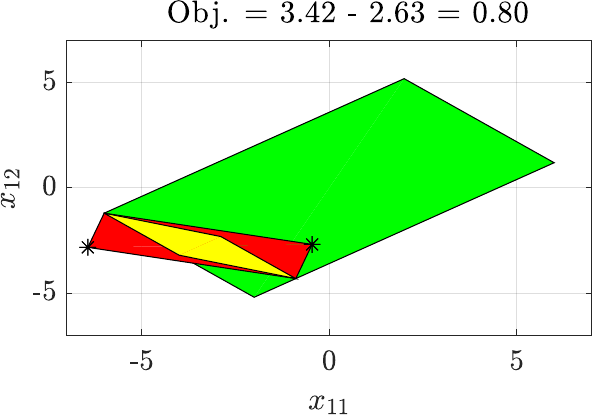}}
	\end{minipage}\hfill
	\begin{minipage}{0.33\textwidth}
		\subfigure[]{\includegraphics[width = \textwidth]{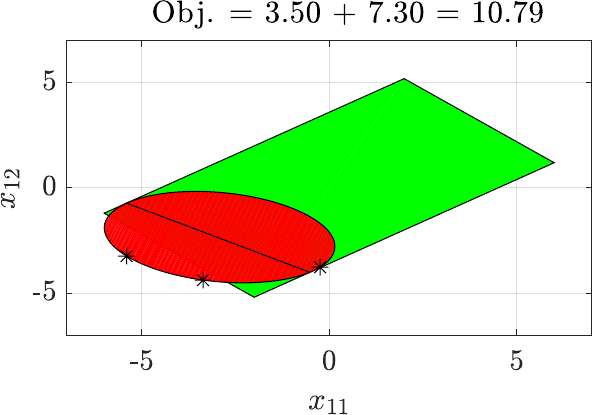}}
	\end{minipage}
	\caption{(green) Feasible set for $ x_1(w) $, (yellow) $ x_1(w) $ for Problem \eqref{problem_toyExample_Decent} with $ \mc S_1 $ as {(a) rectangular}, {(b) rotated by 15 degrees rectangular}, and {(c) rotated by 15 degrees ellipsoid}.}
	\label{fig::rotatedSets}
\end{figure}
To quantify the importance of the primitive set orientation in space, we conducted a second numerical experiment in which we use as primitive sets $ (i) $ a rotated rectangular set which we can independently scale its major and minor axis, and $ (ii) $ a rotated ellipsoid for which the major axis is forced to be $ 1.5 $ times larger than its minor axis. Illustrative examples of the effect in optimal region of such rotated sets are depicted in Fig. \ref{fig::rotatedSets}. We observe that as the shape of the communicated set deviates from the optimal one depicted in Fig. \ref{fig::centralizedSynthesis}(b) then suboptimality and privacy are increased. To clarify this finding, we repeated the simulation experiment for all possible rotations in the range $ [0,\,180] $ degrees of the rectangular and scaled ellipsoid sets. The results are reported in Fig. \ref{fig::costPlot}.  
We observe that if the rotation of the communicated sets matches the one of the set generated by the nested-information policy in Fig. \ref{fig::centralizedSynthesis}(b) then the solution resulting from the proposed decentralized method closely approximates the optimal one. If, however, this is not the case, then the cost considerably deviates and even infeasible instances of Problem \eqref{problem_toyExample_Decent} may appear.
\begin{figure}[h]
	\includegraphics[width=0.6\textwidth]{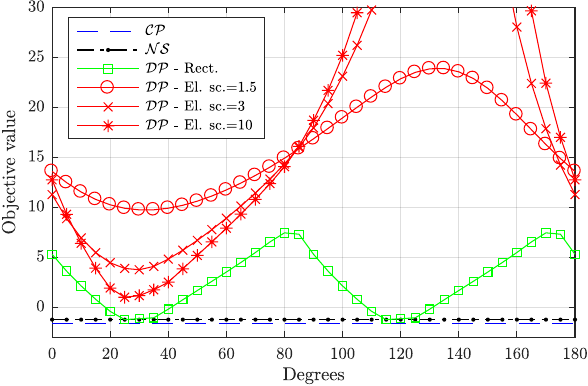}
	\caption{Cost associated with different shapes of communicated sets: (green) rotated rectangular sets, (red) rotated ellipsoids with principal axis ratio of ten.}
	\label{fig::costPlot}
\end{figure}

\subsection{Example 2: Spring-mass-damper}

\begin{figure}[h]
	\centering
	\begin{tikzpicture}
[main node/.style={draw,outer sep=0pt,thick}]

\tikzstyle{spring}=[thick,decorate,decoration={zigzag,pre length=0.3cm,post length=0.3cm,segment length=6}]
\tikzstyle{damper}=[thick,decoration={markings,  
  mark connection node=dmp,
  mark=at position 0.5 with 
  {
    \node (dmp) [thick,inner sep=0pt,transform shape,rotate=-90,minimum width=15pt,minimum height=3pt,draw=none] {};
    \draw [thick] ($(dmp.north east)+(2pt,0)$) -- (dmp.south east) -- (dmp.south west) -- ($(dmp.north west)+(2pt,0)$);
    \draw [thick] ($(dmp.north)+(0,-5pt)$) -- ($(dmp.north)+(0,5pt)$);
  }
}, decorate]

\node[main node] (M1) [minimum width=1cm,minimum height=1cm,xshift=-2.5cm] {$m_4$};
\node[main node] (M2) [minimum width=1cm,minimum height=1cm] {$m_3$};
\node[main node] (M3) [minimum width=1cm,minimum height=1cm,xshift=2.5cm] {$m_2$};
\node[main node] (M4) [minimum width=1cm,minimum height=1cm,xshift=5cm] {$m_1$};


\draw [spring] ($(M1.north east)!0.25!(M1.south east)$) -- ($(M2.north west)!0.25!(M2.south west)$);
\draw [damper] ($(M1.north east)!0.75!(M1.south east)$) -- ($(M2.north west)!0.75!(M2.south west)$);

\draw [spring] ($(M2.north east)!0.25!(M2.south east)$) -- ($(M3.north west)!0.25!(M3.south west)$);
\draw [damper] ($(M2.north east)!0.75!(M2.south east)$) -- ($(M3.north west)!0.75!(M3.south west)$);

\draw [spring] ($(M3.north east)!0.25!(M3.south east)$) -- ($(M4.north west)!0.25!(M4.south west)$);
\draw [damper] ($(M3.north east)!0.75!(M3.south east)$) -- ($(M4.north west)!0.75!(M4.south west)$);

\end{tikzpicture}
	\caption{A chain of four masses connected by springs and dampers.}
	\label{fig::Sim}
\end{figure}
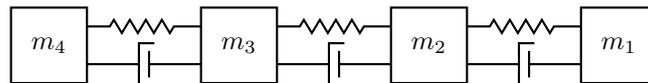
In this numerical study, we consider systems composed of masses that are connected by springs and dampers and arranged in a chain formation, demonstrated in Fig. \ref{fig::Sim}. The values of the masses, spring constants and damping coefficients are chosen uniformly at random from the intervals $ [5,\,10] $kg, $ [0.8,\,1.2] $N/m and $ [0.8,\,1.2] $Ns/m, respectively. We assume that each $ i $-th mass is an individual system with its state vector $ x_{i,t} \in \mb R^2 $ representing the position and velocity deviation from the system's equilibrium state, and its input $ u_{i,t} \in \mb R $ denoting the force applied to the $ i $-th mass. We assume that the states and inputs are constrained such that $ \|x_{i,t}\|_\infty \leq 6 $ and $ \|u_{i,t}\|_\infty \leq 4 $ for all times $ t $. Moreover, we assume that the dynamics of each mass is affected by additive exogenous disturbance $ \xi_{i,t} \in \mb R^2 $ which takes value in the bounded set $ \Xi = \{\xi \in \mb R^2: \|\xi\|_\infty \leq 1\} $.

The prediction control model is obtained by the discretization of the system's continuous dynamics using forward Euler with the sampling time $ 0.1 $s. Although inexact, Euler discretization is chosen as to preserve the distributed structure of the system. On the contrary, the discrete-time simulation model of the system is obtained using the exact zero-order hold discretization method with the sampling time $ 0.1 $s. The objective function of each system is of the form \eqref{eq::objFnc} with $ Q_{i} = \text{diag}(1,0) $ and $ R_i= 0.1 $.

\begin{figure}[h]
	\subfigure[]{\includegraphics[width = 0.4\textwidth]{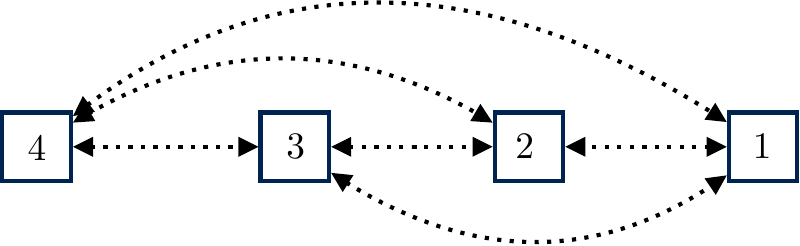}}
	
	\subfigure[]{\includegraphics[width = 0.4\textwidth]{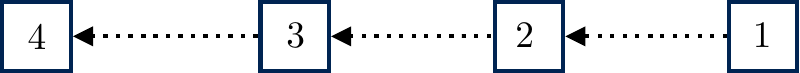}}
	\caption{Communication network associated with (a) \PCd and \PSCd, (b) \PDd for a system of four masses.}
	\label{fig::SMD_Network}
\end{figure}
The dynamics of this interconnected dynamical systems naturally admits a distributed non-nested structure. If one extends the information exchange network as to be nested, then communication needs to be established among all the agents in the system. Hence, the problem formulations for the synthesis of optimal centralized and nested-information controllers, as these are describe in \PCd and \PSCd, are exactly the same. On the contrary, the information exchange network associated with the proposed decentralized method given in \PDd matches the physical coupling network and requires the establishment of communication only between adjacent agents in the system. These information structures are graphically depicted in Fig. \ref{fig::SMD_Network}.

\begin{figure}[h]
	\begin{minipage}{0.48\textwidth}
		\subfigure[]{\includegraphics[width = \textwidth]{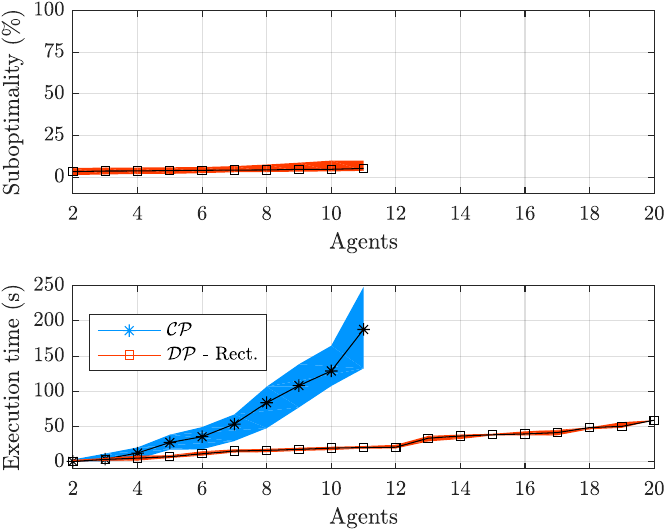}\label{fig::SMD_NT}}
	\end{minipage}\hfill
	\begin{minipage}{0.48\textwidth}
		\subfigure[]{\includegraphics[width = \textwidth]{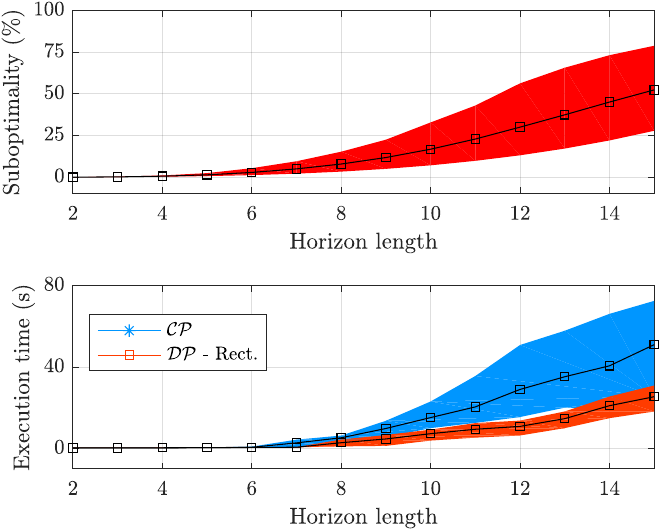}\label{fig::SMD_HT}}
	\end{minipage}
	\caption{Performance comparison as (a) number of agents and (b) horizon length increases. Comparison is performed using 100 Monte-Carlo simulations per system configuration.}
	\label{fig::SMD_NT_HT}
\end{figure}
We now investigate the effect of the horizon length and number of agents on the quality of the solution and the execution time. For each system configuration, we conducted 100 Monte Carlo simulations for uniformly chosen at random initial displacements of the masses in the system. In Fig. \ref{fig::SMD_NT}, we report the effect on this metrics with respect to the number of agents in the system when the horizon length is kept constant to $ T = 8 $. The area in blue is associated with the optimal centralized approach while the one in red with the proposed decentralized one. They both show the range of the respective values over the simulation experiments. We observe that the suboptimality level remains roughly the same with the system size which indicates that the introduced uncertainty as the number of agents increases is dissipated locally by interactions among adjacent agents. On the other hand, the execution time for solving the decentralized robust optimization problem only slightly increases with the number of agents, in contrast to the centralized approach for which the increase is linear. This can be explained considering that the number of decision variables present in \PDd is considerably lower than the one in \PCd. It is interesting to note that \PCd could only be formulated and solved within a reasonable amount of time for a limited number of agents in the network. 

In Fig. \ref{fig::SMD_HT}, we examine the effect of the horizon length on the solution quality and execution time for a system comprising two masses. Here, we observe that the suboptimality of the proposed decentralized method increases linearly with the horizon length. This suggests that the complexity introduced by the evolution of the dynamical system over the time is hard to be accurately enveloped by sets with a predefined orientation. This indicates that conservativeness is being accumulated over time. Moreover, we observe that although the decentralized method is computationally more efficient, the rate of increase in the execution time is the same for both approaches. This was expected since the affine policies make the size of the resulting linear problem to grow polynomially with respect to the horizon length.

\begin{figure}[h]
	\begin{minipage}{0.48\textwidth}
		\subfigure[]{\includegraphics[width = \textwidth]{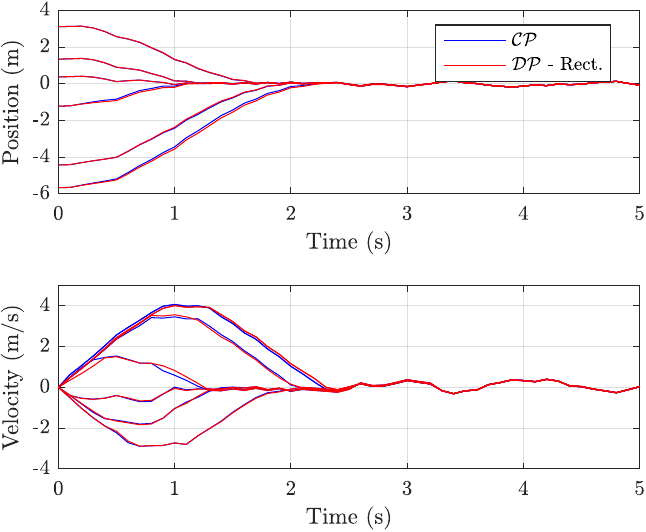}\label{fig::SMD_RH_tr}}
	\end{minipage}\hfill
	\begin{minipage}{0.48\textwidth}
		\subfigure[]{\includegraphics[width = \textwidth]{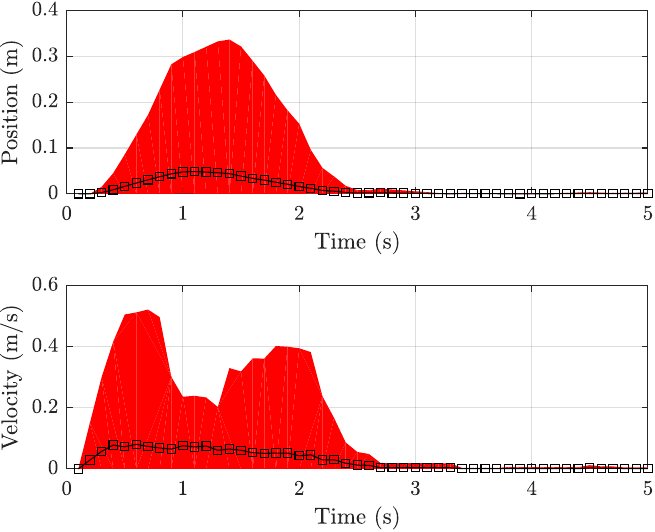}\label{fig::SMD_RH_err}}
	\end{minipage}
	\caption{Comparison of closed-loop trajectories generated by solving \PDd and \PCd; (a) trajectories of 5 randomly initialized simulations, (b) absolute deviation of 100 Monte-Carlo simulation.}
	\label{fig::SMD_RH} 
\end{figure}
Finally, the performance of the system is evaluated on a receding horizon implementation, i.e., the first input resulting from the respective centralized and distributed optimization problem is applied to the exact system dynamics, and the next state is evaluated. We conduct a closed-loop simulation experiment for a system comprising five masses and the prediction horizon of $ T = 15 $.  In Fig. \ref{fig::SMD_RH_tr}, the trajectories generated for the centralized and distributed designs are shown for five different simulation experiments. We observe that these trajectories are very similar, which illustrates the proximity in performance between the centralized and distributed designs. To better quantify the performance comparison between the centralized and distributed designs, we conducted Monte Carlo simulation experiments for uniformly chosen at random initial masses displacements which are also subject to distinct exogenous disturbances realizations. Fig. \ref{fig::SMD_RH_err} shows the error in position and velocity with respect to the simulation time. We observe that as the systems approach their equilibrium state, the errors between the centralized and decentralized approach converges to zero. We note, however, that in all instances this error in position and velocity is fairly small, although the respective open-loop error is high (c.f., Fig \ref{fig::SMD_HT}), which indicates the efficacy of the proposed method in closed-loop simulations.

\subsection{Example 3: Supply chain with quantity flexibility contracts}

\begin{figure}[h]
	\centering	
	\includegraphics[width = 0.8\textwidth]{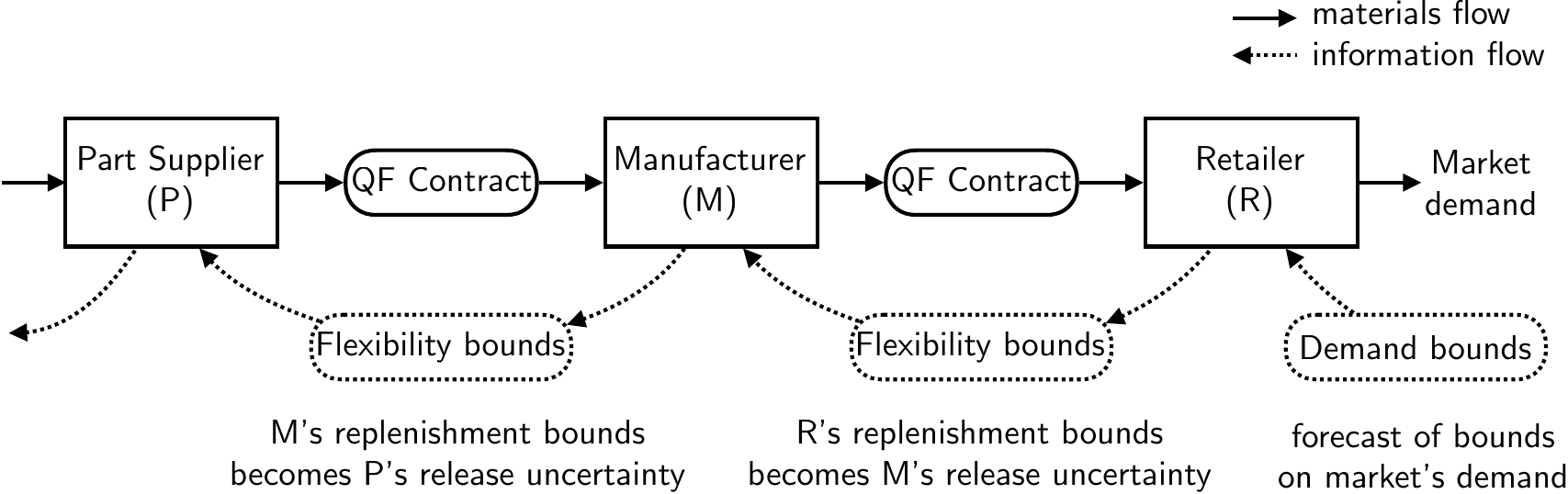}
	\caption{Modern supply chains operated in a distributing decision making authority.}
	\label{fig::QFC_v2}
\end{figure}
We now exhibit the performance of the proposed method as a contract design mechanism for supply chains operated in a decentralized fashion and we investigate its performance on various physical coupling topologies, e.g., chain, star, ring and mesh structures. As a running example, we adopt the decentralized operation of modern supply chains using quantity flexibility (QF) contracts, as this is described in \cite{Tsay1999}. Modern supply chains naturally operate in a distributing decision making authority since multiple sites worldwide work together to deliver product. In this distributed decision making context, each manufacturer (supplier) knows only the schedule of desired replenishments provided by its immediate retailer (manufacturer), and is only concerned with its own cost performance. To avoid ``mutual deception'' situations (e.g., some buyers inflate demand only to later disavow any undesired product \cite{Lee1997}) which increase the uncertainties and costs in the system \cite{Magee1967,Lovejoy1998}, a commonly used approach in the industry is to introduce the QF contract as a method for coordinating materials and information flows in distributed supply chains. This type of contract discourages the customer from overstating its needs by allowing a maximum upside revision of its scheduled replenishment. In this context, the supplier is obligated to cover any requests that remain within these limits. We graphically depict in Fig. \ref{fig::QFC_v2} the operation of a single-product, serial supply chain using QF contracts.

For each $ i $ agent, i.e., retailer, manufacturer or supplier, we consider inventory dynamics given as
\begin{equation}
I^i_{t+1,p} = I^i_{t,p} + R^i_{t,p} -D^i_{t,p} \text{ for }  p = 1,\ldots,P
\end{equation}
where $ I^i_{t,p} $ is the inventory stock, $ R^i_{t,p} $ is the replenishment and $ D^i_{t,p} $ is the customer demand for the $ p $-th out of $ P $ products. For manufacturers and/or suppliers, the demand originates from the replenishment schedule of another manufacturer or a retailer. If the $ i $-th agent is a retailer, then the product demand originates from the market and we assume that at each period $ t $ is given as 
\begin{equation}
D_{t,p}^i = \left \lbrace\begin{array}{ll}
2 + \sin\left(2\pi \dfrac{t}{T-1}\right) + \dfrac{1}{k} \sum \limits_{i=1}^{k} \Phi_{p,i}^i \xi_{t}^i & \text{for } p \text{ even},\\
2 + \cos\left(2\pi \dfrac{t}{T-1}\right) + \dfrac{1}{k} \sum \limits_{i=1}^{k} \Phi_{p,i}^i \xi_{t}^i & \text{for } p \text{ odd},
\end{array}\right.
\end{equation}
where $ \xi_{t,p}^i \in \Xi_{i,p} = [-1,1] $ and the factor loading coefficients $ \Phi_{p,i}^i $ are chosen uniformly at random from $ [-1,1] $. By construction, the product demands thus satisfy $ D_{t,p}^i \in [0,4] $ for all $ t \in \mc T $. To this end, we model the process of product making, i.e., combination of different materials, as follows 
\begin{equation}
R_{t,p}^i =  \sum \limits_{j=1}^{\mc N_i} \Psi_{p,j}^i D_{t,p}^j + w_{t,p}^i \text{ for } p = 1,\ldots,P
\end{equation}
where $ \Psi_{p,j}^i $ are chosen uniformly at random from $ [0,1] $ as to satisfy $ \sum \limits_{j=1}^{\mc N_i} \Psi_{p,j}^i = 1 $. Moreover, $ w_{t,p}^i \in \mc W_i = [-0.2,0] $ is a random variable that captures production delays and materials loss. Any excess demand is backlogged at a unit cost of $ c_B $ per period, and any excess inventory incurs a unit holding cost of $ c_H $ per period. The objective is to determine an ordering policy that minimizes the worst-case sum of backlogging and inventory holding costs over all anticipated demand realizations given as,
\begin{equation}
\max \limits_{\bs \xi, \bs w } \sum_{t \in \mc T} \sum_{p \in P_i} c_H \big[ I^i_{t,p}\big]_+ + c_B \big[-I^i_{t,p}\big]_+
\end{equation}
where $ c_B, c_H $ are chosen uniformly at random in the interval $ [0,2] $, and the maximum operator $ [\cdot]_+ $ can be easily removed using the epigraph representation.

\begin{figure}[h]
	\includegraphics[width=0.6\textwidth]{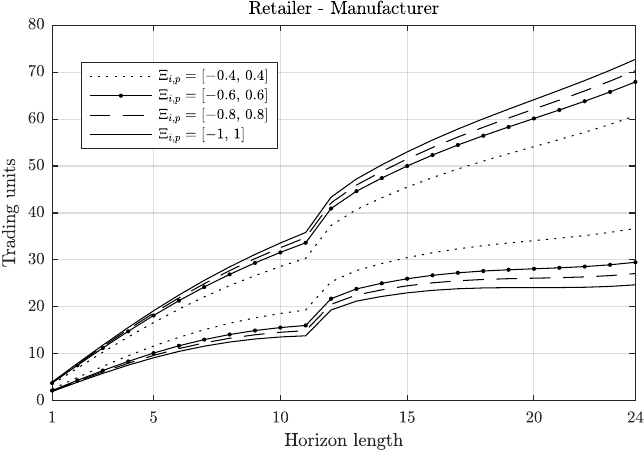}
	\caption{Effect of uncertainty on retailer-manufacturer communicated sets over a 24 stages horizon.}
	\label{fig::SC_UndEff}
\end{figure}
In the first simulation experiment, we investigate for the structure depicted in Fig. \ref{fig::QFC_v2} how the increase on the market demand uncertainty affects the bounds communicated from retailers to manufactures. The simulation configuration comprises one retailer, one manufacturer and one supplier and we symmetrically increase the size of the uncertainty set $ \Xi_{i,p} $ from $ 40 \% $ to $ 100 \% $ of its original size. We report the results in Fig. \ref{fig::SC_UndEff} where we observe that the size of the communicated bounds increases as the size of the uncertainty and the horizon length increase. Although, this was expected it is interesting to visualize how the bounds adapt on the demand profile pattern, highlighting the cooperative nature of the approach which strives to minimize the uncertainty mitigating backwards in the supply chain.

We now investigate the performance of the proposed approach on different supply chains topologies which vary on size. In particular, we focus our attention on the chain, ring, mesh and star topologies, graphically depicted in Fig. \ref{fig::SC_Topologies}.
\begin{figure}[h]
	\begin{minipage}[b]{0.3\textwidth}
		\centering
		\subfigure[]{\includegraphics[width = \textwidth]{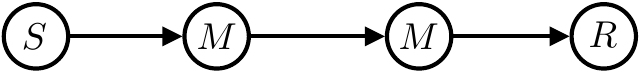}}
		
		\subfigure[]{\includegraphics[width = \textwidth]{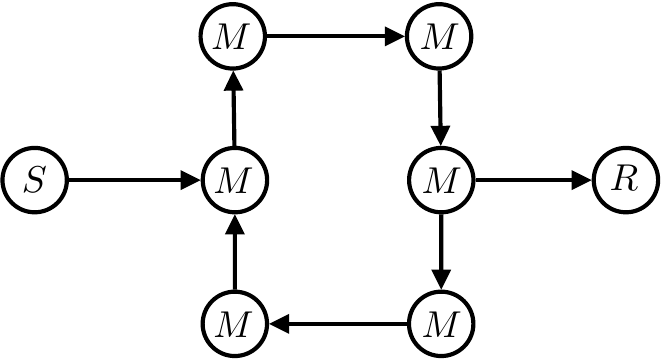}}
	\end{minipage}\hspace*{1cm}
	\begin{minipage}[b]{0.3\textwidth}
		\centering
		\subfigure[]{\includegraphics[width = 0.8\textwidth]{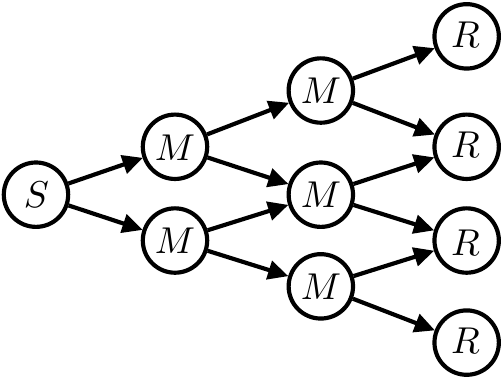}}
		
		\subfigure[]{\includegraphics[width = \textwidth]{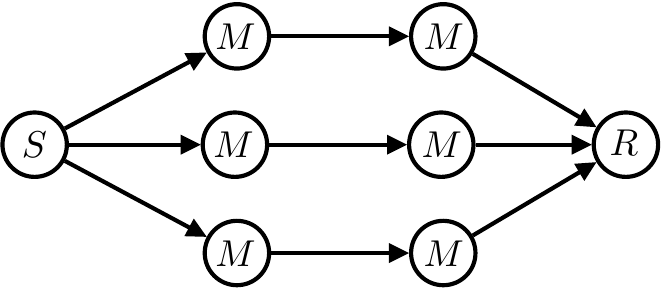}}
	\end{minipage}
	\caption{Coupling topologies under consideration: (a) chain, (b) mesh, (c) ring, (d) star. }
	\label{fig::SC_Topologies}
\end{figure}

\begin{figure}[b]
	\includegraphics[width=0.6\textwidth]{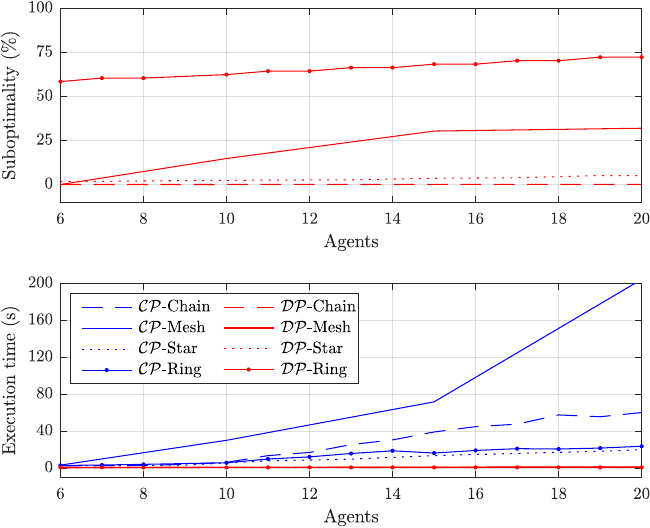}
	\caption{Performance comparison for different coupling topologies as the number of agents in the network increases.}
	\label{fig::SC_Comparison}
\end{figure}
To assess the efficacy of the proposed method, we compare the objective value and execution time of the decentralized solution given by \PDdATC solved with rectangular primitive sets, and the solution generated by the centralized \PCd. We report the results of this comparison for networks admitting from 6 to 20 agents in Fig. \ref{fig::SC_Comparison}. 
We observe that for all structures under investigation the increase on the size of the agents in the network only slightly affects the suboptimality increase. For structures which are loosely coupled, as the chain and star, the suboptimality of the decentralized solution is close to zero. On the other hand, the worst-performance is observed for the ring structure since due to its cyclic nature its hard to obtain a decentralized solution which approximates the centralized one and is based only on neighbors communication. The large benefit of the proposed decentralized approach is however identified on the execution time needed to generate its solution. As shown in Fig. \ref{fig::SC_Comparison}, even for small networks with loose coupling the decentralized approach can be 50 times faster than the centralized, while in the extreme case of a large mesh structure this efficiency gap can grow up to 200 times faster.

\section{Conclusions} \label{sec::Conclusion}
This paper presents a decentralized control framework for the problem of cooperatively managing the operation of large-scale networks of constrained dynamical systems. The proposed method requires each system to communicate to its neighboring systems bounds on the evolution of its states. This minimal communication provides a certain degree of privacy to each agent since the exact characteristics of the dynamics within each system are not revealed. Moreover, the method is suitable to manage problems involving large-scale systems since its underlying minimum communication scheme preserves the original decoupled structure of the problem. To optimize the size of the communicated bounds, a non-convex infinite-dimensional problem is formulated. This computationally intractable optimization problem is approximated by a convex finite-dimensional one, using methods from robust optimization. It is shown that these approximations retain the decoupled structure of the problem making it amendable to distributed computation algorithms. In the numerical study, it is shown that the proposed decentralized method achieves highly computationally efficient solutions even for networks involving a large number of agents. Depending on the structure of underlying physical coupling network and the form of the communication bounds, the proposed method is shown to generate solutions which closely approximate those of the centralized formulation.

\appendix
\section{Proofs of Propositions and Theorems}
\setcounter{equation}{0}
\renewcommand{\theequation}{\Alph{section}.\arabic{equation}}

\begin{proof}[\large \bf Proof of Theorem \ref{thm::1}]

	We show that every feasible solution of \PSCd is a feasible solution to \PCd. Let $ \bs \Phi_i $ for all $ i \in \mc M $ be any feasible policy in \PSCd. Starting with $ \chi_{i,1} = x_{i,1}$, the state of agent $ i $ at time $ t $ are given as
	\begin{equation}\label{eq::thm1::eq1}
		\begin{array}{@{}r@{\;}l}
		x_{i,t} &=\displaystyle A_{i,1}^t x_{i,1} + \sum_{\tau=1}^{t-1} \Big(A_{i,\tau+1}^t B_{i,\tau} [\chi_{j,\tau}(\bs \xi_{\oNj}^{\tau-1})]_{j\in \mc N_i} + A_{i,\tau+1}^t D_{i,\tau} \Phi_{i,\tau}(\bs \xi_{\oNi}^{\tau-1}) + A_{i,\tau+1}^t E_{i,\tau} \xi_{i,\tau} \Big) \\
		&=: \chi_{i,t}(\bs \xi_{\oNi}^{t-1})
		\end{array}
	\end{equation}
	where $ A_{i,\tau}^t = A_{i,\tau}A_{i,\tau+1}\dots A_{i,t-1} $ for $ \tau<t $ and $ A_{i,t}^t = I $. The last implication follows from the fact that $ \ol{\mc N}_i \supseteq \ol{\mc N}_j $ for all $ j \in \mc N_i $ since the network admits a partially nested structure. Given \eqref{eq::thm1::eq1}, it is easy to verify that for each agent $ i $ its dynamics and constraints in \PSCd only depend on $ \bs \xi_{\oNi} $. Hence, any feasible solution to \PSCd is also feasible to the following optimization problem:
	\begin{equation}\label{eq::thm1::eq2}
	\begin{array}{l}
	\text{minimize}  \;\; \displaystyle\sum\limits_{i = 1}^M \max\limits_{\bs \xi \in \Xi} J_i(\bs x_i,\bs u_i) \\
	\!\!\!\left.\begin{array}{r@{\;}l@{}}
	\text{subject to}& \bs \phi_{i}(\cdot) \in \mc {SC}(\bs \xi_\oNi),\;\bs u_i = \bs \Phi_{i}(\bs \xi_\oNi)\\
	& \bs x_{i} = f_{i}\big(\bs x_{\Ni}, \bs u_{i}, \bs \xi_{i}\big)\\
	& (\bs x_{i}, \bs u_{i}) \in \mc O_i
	\end{array}\right \rbrace \forall \bs \xi_\Mu \in \Xi_\Mu,\;
	\forall i \in \mc M
	\end{array}
	\end{equation}
	Additionally, they achieve the same objective value since they share the same objective function. This shows equivalence of \PSCd and Problem \eqref{eq::thm1::eq2}. The relation between \PCd and \PSCd stated in the theorem now follows immediately since the two problems share the same constraints and objective function, and $\mc {SC}(\bs \xi_\oNi)\subseteq\mc {SC}(\bs \xi)$ for all $i\in\mc M$.
\end{proof}

\begin{proof}[\large \bf Proof of Theorem \ref{thm::2}]

	The statement is proved by induction using similar theoretical tools to \cite[Prop. 2.1]{HGK:2011b}. The statement holds for $ t = 1 $ since the initial state, $ x_{i,1} $, is known for every $ i \in \mc M $; therefore, functions $ \psi_{i,1}(x_{i,1}, \bs \zeta_{\Ni,1}) $ and $ \Psi_{i,1}(\bs \zeta_{\Ni,1}) $ can always be constructed such
	\begin{equation}
		\psi_{i,1}(x_{i,1}, \bs \zeta_{\Ni,1}) = \Psi_{i,1}(\bs \zeta_{\Ni,1})
	\end{equation}
	
	Assume now that the statement holds for all $ 1 < \tau \leq t-1 $, i.e., there exist policies $ \psi_i(\cdot) $ and $ \Psi_i(\cdot) $ such that $ \psi_{i,\tau}(\bs x_i^{\tau}, \bs \zeta_\Ni^{\tau}) = \Psi_{i,\tau}(\bs \xi_i^{\tau-1}, \bs \zeta_\Ni^{\tau}) $. In the sequel, we show that the statement also holds for $ \tau = t $. From \eqref{eq::stateDynamics}, we have that,
	\begin{equation}\label{eqqv1}
	\begin{array}{r@{}l}
	x_{i,t} &=\displaystyle A_{i,1}^t x_{i,1} + \sum_{\tau=1}^{t-1} \Big( A_{i,\tau+1}^t B_{i,\tau} \bs \zeta_{\Ni,\tau} + A_{i,\tau+1}^t D_{i,\tau} \Psi_{i,\tau}(\bs \xi_i^{\tau-1}, \bs \zeta_\Ni^\tau) +  A_{i,\tau+1}^t E_{i,\tau} \xi_{i,\tau} \Big) \\
	&=: \chi_{i,t}(\bs \xi_{i}^{t-1}, \bs \zeta_{\Ni}^{t-1})
	\end{array}
	\end{equation}
	where $ A_{i,\tau}^t = A_{i,\tau}A_{i,\tau+1}\dots A_{i,t-1} $ for $ \tau<t $ and $ A_{i,t}^t = I $. Moreover, it holds that,
	\begin{equation}\label{eqq1v1}
	\begin{array}{r@{}l}
	\xi_{i,t-1} &= E_{i,t-1}^{+}\big(x_{i,t} - A_{i,t-1} x_{i,t-1} + B_{j,t-1} \bs \zeta_{\Ni,t-1} + D_{i,t-1} \psi_{i,t-1}(\bs x_i^{t-1}, \bs \zeta_\Ni^{t-1})\big) \\
	&=: \rho_{i,t}(\bs x_{i}^{t}, \bs \zeta_\Ni^{t-1})
	\end{array}
	\end{equation}
	where $ E^+_{i,t} := (E_{i,t}^\top E_{i,t})^{-1}E^\top_{i,t} $ is the left inverse of $E_{i,t}$ since it is full rank.
	
	The relation \eqref{eqqv1} implies that given a feasible policy $ \psi_{i,t}(\cdot) $ for \PDs, we can construct a feasible policy for \PDd as,
	\begin{equation}\label{eqq2v1}
	\psi_{i,t}(\bs x_i^t, \bs \zeta_\Ni^t) = \psi_{i,t}(\chi_i^t(\bs \xi_{i}^{t-1},\bs \zeta_{\Ni}^t), \bs \zeta_\Ni^t):= \Psi_{i,t}(\bs \xi_i^{t-1}, \bs \zeta_\Ni^t)
	\end{equation}
	The claim follows from the fact that the composition of continuous differentiable functions is a continuous differentiable function. Hence, the policy $ \psi_{i,t}(\cdot) $ will also be feasible in \PDd since the two problems have the same pointwise constraints. Additionally, they achieve the same objective value since they share the same objective function. 
	
	Similarly, the relation \eqref{eqq1v1} implies that given a feasible policy $ \Psi_{i}(\cdot) $ for \PDd, we can construct a feasible policy for \PDs as,
	\begin{equation}\label{eqq3v1}
	\Psi_{i,t}(\bs \xi_i^{t-1}, \bs \zeta_\Ni^t) = \Psi_{i,t}(\rho_i^t(\bs x_{i}^{t},\bs \zeta_\Ni^{t-1}), \bs \zeta_\Ni^t) := \psi_{i,t}(\bs x_i^t, \bs \zeta_\Ni^t)
	\end{equation}
	The claim follows from the fact that the composition of continuous differentiable functions is a continuous differentiable function. Hence, the policy $ \Psi_{i,t}(\cdot) $ will also be feasible in \PDs since the two problems have the same pointwise constraints. Additionally, they achieve the same objective value since they share the same objective function. 
\end{proof}

\begin{proof}[\large \bf Proof of Theorem \ref{thm::3}]	
	We show that every feasible solution of \PDd is feasible in \PSCd. Let $ (\bs \Psi_i, \mc X_i) $ for all $ i \in \mc M $ be feasible in \PDd. Since the state of agent $i$ evolve according to \eqref{eq::stateDynamics}, we can conclude that at time $t$ we have	 
	\begin{equation}
	\begin{array}{r@{}l}
	x_{i,t} &=\displaystyle A_{i,1}^t x_{i,1} + \sum_{\tau=1}^{t-1} \Big( A_{i,\tau+1}^t B_{i,\tau} \bs \zeta_{\Ni,\tau} + A_{i,\tau+1}^t D_{i,\tau} \Psi_{i,\tau}(\bs \xi_i^{\tau-1}, \bs \zeta_\Ni^\tau) + A_{i,\tau+1}^t E_{i,\tau} \xi_{i,\tau}  \Big)\\
	&=: \chi_{i,t}(\bs \xi_{i}^{t-1}, \bs \zeta_{\Ni}^{t-1})
	\end{array}
	\end{equation}
	where where $ A_{i,\tau}^t = A_{i,\tau}A_{i,\tau+1}\dots A_{i,t-1} $ for $ \tau<t $ and $ A_{i,t}^t = I $. Note that $ \bs\chi_{i}(\bs \xi_{i},\bs \zeta_{\Ni}) = [\chi_{i,t}(\bs \xi_{i}^{t-1}, \bs \zeta_{\Ni}^{t-1})]_{t \in \mc T_+} \in \mc X_i $ for all $ \bs \xi_i \in \Xi_i $ and $ \bs \zeta_{\Ni} \in \mc X_\Ni $ due to the feasibility of \PDd. To show that $ \bs \Psi_i $ is feasible in \PSCd, we first construct the state of agent $i$ which evolves according to $\bs x_{i} = f_{i}\big(\bs  x_{\Ni}, \bs \Psi_{i}(\bs{\xi}_i, \bs \zeta_\Ni), \bs \xi_{i}\big)$. Starting with $\hat \chi_{i,1} = x_{i,1}$, we have that
	\begin{equation}\label{eq::c1}
	\begin{array}{r@{}l}
	x_{i,t} &=\displaystyle A_{i,1}^t x_{i,1} + \sum_{\tau=1}^{t-1} \Big(A_{i,\tau+1}^t B_{i,\tau} [\widehat \chi_{j,\tau}(\bs \xi_{\oNj}^{\tau-1}]_{j\in \mc N_i} + A_{i,\tau+1}^t D_{i,\tau} \Psi_{i,\tau}(\bs \xi_i^{\tau-1}, \bs \zeta_{\Ni}^{\tau}) + A_{i,\tau+1}^t E_{i,\tau} \xi_{i,\tau} \Big)\\
	&=\displaystyle A_{i,1}^t x_{i,1} + \sum_{\tau=1}^{t-1} \Big(A_{i,\tau+1}^t B_{i,\tau} [\widehat \chi_{j,\tau}(\bs \xi_{\oNj}^{\tau-1}]_{j\in \mc N_i} + A_{i,\tau+1}^t D_{i,\tau} \Psi_{i,\tau}(\bs \xi_i^{\tau-1}, [\widehat{\bs \chi}_{j}^\tau (\bs \xi_{\oNj}^{\tau-1}]_{j\in \mc N_i}) + A_{i,\tau+1}^t E_{i,\tau} \xi_{i,\tau} \Big)\\
	& = \chi_{i,t}(\bs \xi_i^{t-1},[\widehat \chi^{t-1}_{j}(\bs \xi_{\oNj}^{t-2})]_{j\in \mc N_i})\\
	& =: \widehat\chi_{i,t}(\bs \xi_{\oNi}^{t-1}).
	\end{array}
	\end{equation}
	where the implication follows because $\bs{\widehat\chi}_{i}([\bs \xi_{j}^{t-1}]_{j\in \oNi}) = [\widehat\chi_{i,t}(\bs \xi_{\oNi}^{t-1})]_{t \in \mc T_+}\in\mc X_i$ for all $ \bs \xi_\oNi \in \Xi_\oNi$ and $ i \in \mc M $.
	For each $i\in \mc M$, we consider the decision $\bs{\hat \Phi}_i\in\mc {SC}(\bs \xi_\oNi)$ defined through
	\begin{equation}\label{eq::c2}
	\widehat \Phi_{i,t}(\bs \xi_{\oNi}^{t-1}):= \Psi_{i,t}(\bs \xi_i^{t-1},[\widehat \chi^{t}_{j}(\bs \xi_{\oNj}^{t-1})]_{j\in \mc N_i})
	\end{equation}
	Notice that \eqref{eq::c2} defines a valid policy construction since $\bs{\widehat\chi}_{i}([\bs \xi_{j}^{t-1}]_{j\in \oNi})\in\mc X_i$ for all $ \bs \xi_\oNi \in \Xi_\oNi$. It remains to show that $\bs \Psi_{i}$ is feasible also for the constraints of \PSCd. We do so using deduction, as follows:  
	\begin{equation}
	\begin{array}{rll}
	& \big(\bs\chi_{i}(\bs \xi_{i},\bs \zeta_{\Ni}), \bs \Psi_{i}(\bs{\xi}_i,\bs \zeta_{\Ni})\big) \in \mc O_i, & \forall \bs \xi_i \in {\Xi}_i, \forall \bs \zeta_\Ni \in \mc X_\Ni,\\
	\implies & \big(\bs\chi_{i}(\bs \xi_{i},[\bs{\widehat\chi}_{j}(\bs \xi_\oNj)]_{j \in \mc N_i}), \bs \Psi_{i}(\bs{\xi}_i,[\bs{\widehat\chi}_{j}(\bs \xi_\oNj)]_{j \in \mc N_i})\big)  \in \mc O_i, & \forall \bs \xi_\oNi \in {\Xi}_\oNi, \\
	\implies & \big(\bs{\widehat\chi}_{i}(\bs \xi_\oNi), \bs{\widehat\Phi}_{i}(\bs \xi_\oNi)\big)  \in \mc O_i, & \forall \bs \xi_\oNi \in {\Xi}_\oNi,
	\end{array}
	\end{equation}
	The first implication follows from \eqref{eq::c1} and the fact that $\bs{\widehat\chi}_{i}(\bs \xi_\oNi) \subseteq \mc X_i$ for all $\bs{\xi}_\oNi\in\Xi_\oNi$, while the second implication follows from \eqref{eq::c1} and \eqref{eq::c2}.
	Finally, this feasible solution attains a value for the objective function of \PDd which is equal or larger than the value attained for the objective function of \PSCd, that is:
	\begin{equation*}
	\begin{array}{@{\,}r@{\,}l@{\,}}
	&\sum\limits_{i = 1}^M \max\limits_{\bs \xi_i \in \Xi_i, \bs \zeta_{\Ni} \in \mc X_{\Ni}} 
	J_i\big(\bs\chi_{i}(\bs \xi_{i},\bs \zeta_{\Ni}), \bs \Psi_{i}(\bs{\xi}_i,\bs \zeta_{\Ni})\big)\\
	\geq& \sum\limits_{i = 1}^M \max\limits_{\bs \xi_\oNi \in \Xi_\oNi}\; J_i\big(\bs\chi_{i}(\bs \xi_{i},[\bs{\widehat\chi}_{j}(\bs \xi_\oNj)]_{j \in \mc N_i}), \bs \Psi_{i}(\bs{\xi}_i,[\bs{\widehat\chi}_{j}(\bs \xi_\oNj)]_{j \in \mc N_i})\big) \\
	=& \sum\limits_{i = 1}^M \max\limits_{\bs \xi_\oNi \in \Xi_\oNi}\; J_i\big(\bs{\widehat\chi}_{i}(\bs \xi_\oNi), \bs{\widehat\Phi}_{i}(\bs \xi_\oNi)\big)
	\end{array}
	\end{equation*}
	where again the first implication follows from \eqref{eq::c1} and the fact that $\bs{\widehat\chi}_{i}(\bs \xi_\oNi) \subseteq \mc X_i$ for all $\bs{\xi}_\oNi\in\Xi_\oNi$, while the second implication follows from \eqref{eq::c1} and \eqref{eq::c2}.
\end{proof}

\begin{proof}[\large \bf Proof of Theorem \ref{thm::4}]
	We consider the 3-Satisfiability problem (3-SAT) for a set $ N = \{1,\ldots,n\} $ literals and $ m $ clauses, which seeks to find a solution $ x \in \{0,1\}^{n} $ that satisfies
	\begin{equation}\label{3sat}
	x_{i,1} + x_{i,2} + (1-x_{i,3}) \ge 1, \forall i = 1,\ldots,m
	\end{equation}
	where $ x_{i,1}, x_{i,2}, x_{i,1}  \in \{0,1\} $ are auxiliary variables that allow us to reformulate the disjunction of $ n $ literals into a conjunction of $ m $ clauses \cite{Cook1971}.
	
	We also consider the following robust optimization problem with decision-dependent uncertainty set:
	\begin{equation}
	\label{problem_3sat}
	\begin{array}{@{}l@{}}
	\min \sum_{i=1}^m \max(-\alpha_i)  \\
	\left.\begin{array}{r@{\;}l@{}}
	\st &  x_{i,1}, x_{i,2} \ge 0, \\
	& \alpha_i \in \mc X_i(x_{i,1}, x_{i,2}) = \{\alpha_i: \alpha_i \ge x_{i,1}, \alpha_i \ge x_{i,2}, \alpha_i \leq 1\},\\
	& (1-x_{i,3}) \leq \alpha_i, \;\forall \alpha_i \in \mc X_i(x_{i,1}, x_{i,2}),\\
	& x_{i,3} \ge 0, x_{i,3} \leq 1,
	\end{array}\right \rbrace \begin{array}{@{}l}
	\forall i = 1,\ldots,m.
	\end{array}
	\end{array}
	\end{equation}
	where $ \alpha_i \in \mb R $ are auxiliary decision variables. If we assign the decision variables $ x_{i,1},x_{i,2}, \alpha_i $ to agent $ i_1 $ and $ x_{i,3} $ to agent $ i_2 $ with $ i = 1,\ldots,m $ then it is easy to verify that Problem \eqref{problem_3sat} is an instance of \PDdCCa involving $ 2m $ agents. We now prove using similar theoretical tools as those developed in \cite[Thm. 1]{Nohadani2016} that the optimal value of Problem \eqref{problem_3sat} is $ -m $ if and only if the 3-SAT problem in \eqref{3sat} has a solution. We do so as follows: $ (i) $ we assume that the 3-SAT problem has a solution, i.e., there exist $ x \in \{0,1\}^{n} $ such that \eqref{3sat} is satisfied. This implies that either of the terms $ x_{i,1} $, $ x_{i,2} $ or $ 1-x_{i,3} $ is one, hence $ \alpha_i = 1 $ for all $ i = 1,\ldots, M $ due to the constraints in Problem \eqref{problem_3sat}. This leads to an optimal objective value of at least $ -m $ for Problem \eqref{problem_3sat}; $ (ii) $ We now prove by contradiction that the optimal objective of Problem \eqref{problem_3sat} can not be smaller than $ -m $. If so then there must exists $ \alpha_i $ such that $ \alpha_i > 1 $ which contradicts the constraint $ \alpha_i \leq 1 $. Hence, if Problem \eqref{problem_3sat} is solved to optimality, i.e., $ \sum_{i=1}^m \max(-\alpha_i) = -m $, then $ \alpha_i = 1 $ for all $ i = 1, \ldots, m $. In this case, at least one of the terms $ x_{i,1} $, $ x_{i,2} $ or $ 1-x_{i,3} $ is one otherwise the uncertainty set $ \mc X_i $ is not a singleton and an $ \alpha_i = \max\{x_{i,1},x_{i,2},1-x_{i,3}\} < 1 $ is also valid which would then give an optimal objective $ \sum_{i=1}^m \max(-\alpha_i) > -m $. This however clearly contradicts the assumption that Problem \eqref{problem_3sat} is solved to optimality. 
	
	Therefore, solving the optimization Problem \eqref{problem_3sat} is equivalent to trying to find a feasible solution for the 3-SAT problem. However, as shown in \cite{Cook1971} the feasibility of the 3-SAT problem is known NP-hard, hence the optimization Problem \eqref{problem_3sat} is also NP-hard.
\end{proof}

\begin{proof}[\large \bf Proof of Proposition \ref{prop::nCc}]
	The recession cone of a set $ \mc S_i $ is defined as $ \textrm{recc}(\mc S_i) = \{\bs \nu_i \in \mb R^{n_x^i}: s_i + \lambda \nu_i \in \mc S_i, \forall s_i \in \mc S_i, \,\lambda \ge 0 \} $ \cite{Gorissen2014}. The fact that $ \mc S_i $ is bounded implies that the recession cone of $ \mc S_i $ is empty, i.e., $ \textrm{recc}(\mc S_i) = \{0\} $. We now show that,
	\begin{equation*}
	\mc X_{\FS} = \Big\{(\bs x_i, y_i, z_i)\,:\, \exists \bs s_i \in\mathbb{R}^{N_{x,i}} \text{ s.t. }\bs x_{i} = \sum_{k=1}^{K_i} y_{i,k} P_{i,k} \bs s_{i} + z_{i},\; G_{i,k} P_{i,k} \bs s_{i} \preceq_{\mc K_{i,k}}  g_{i,k},\; k=1,\ldots,K_i \Big\} 
	\end{equation*}
	is equivalent to
	\begin{equation*}
	\wh{\mc X}_{\FS} =\Big\{(\bs x_i, y_i, z_i)\,:\, \exists \bs  \nu_{i,k}\in\mathbb{R}^{N_{x,i}} \text{ s.t. } \bs x_{i} = \sum_{k=1}^{K_i} P_{i,k} \bs \nu_{i,k} + z_{i},\; G_{i,k} P_{i,k} \bs \nu_{i,k}  \preceq_{\mc K_{i,k}}  y_{i,k}  g_{i,k},\;k=1,\ldots,K_i \Big\}
	\end{equation*}
	It is easy to verify that this deduction also holds true in case that $ y_{i,k} $ are positive scalar by using the substitution $\bs{\nu}_{i,k} = y_{i,k} \bs{s}_i$. In the case that any $ y_{i,k} = 0 $ then it remains to show that the only feasible solution is $\bs{\nu}_{i,k} = 0 $ so that the equality $\bs{\nu}_{i,k} = y_{i,k} \bs{s}_i$ to hold. Assume that this is not the case, i.e., there exist $ \bs{\nu}_{i,k} \neq 0 $. Then, $ \bs{\nu}_{i,k} \in \textrm{recc}(S_i) $ which means that the $ \mc S_i $ recedes in the direction of  $ \bs{\nu}_{i,k} $. However, this is a contradicts the boundedness of $ \mc S_i $. 
\end{proof}

\subsection*{Preliminaries for the proof of Theorem \ref{thm::5}}
The following lemma establishes constraint satisfaction between \PDdATC and \PDdAT.
\begin{lemma}\label{lem::1}
	Given $ y_i $ and $ z_i $ such that $ \wh{\mc X}_i(y_i, z_i) $, then for any two functions $ f_{i,t} $ and $ g_{i,t} $, it holds:
	\begin{equation}\label{eq::op}
	\begin{array}{r@{\,}ll}
	& f_{i,t}(\bs \zeta_\Ni^t, \bs \xi_i^t) \leq 0,& \forall \bs \zeta_\Ni \in \wh{\mc X}_\Ni(y_\Ni, z_\Ni), \,\forall \bs \xi_i \in \Xi_i, \\
	\Rightarrow & f_{i,t}\big([L_j^t(\bs s_j^t)]_{j\in \mc N_i}), \bs \xi_i^t\big) \leq 0, &\forall \bs s_\Ni \in \mc S_\Ni,\,\forall \bs \xi_i \in \Xi_i,
	\end{array}
	\end{equation}
	and
	\begin{equation}\label{eq::opInv}
	\begin{array}{r@{\,}ll}
	& g_{i,t}(\bs s_\Ni^t, \bs \xi^t)\leq 0,& \forall \bs s_\Ni \in \mc S_\Ni,\,\forall \bs \xi_i \in \Xi_i, \\
	\Rightarrow & g_{i,t}\big([R_j^t({\bs \zeta_j^t})]_{j\in \mc N_i}, \bs \xi_i^t\big) \leq 0, & \forall \bs \zeta_\Ni \in \wh{\mc X}_\Ni(y_\Ni, z_\Ni), \,\forall \bs \xi_i \in \Xi_i.
	\end{array}
	\end{equation}
\end{lemma}
\begin{proof}
	We prove \eqref{eq::op} by contradiction. Assume that $ f_{i,t}(\bs \zeta_\Ni^t, \bs \xi_i^t) \leq 0, \forall \bs \zeta_\Ni \in \wh{\mc X}_\Ni(y_\Ni, z_\Ni), \,\forall \bs \xi_i \in \Xi_i, $ and there exist $ \bs s_\Ni \in \mc S_\Ni $ such that $ f_{i,t}\big([L_j^t(\bs s_j^t)]_{j\in \mc N_i}), \bs \xi_i^t\big) > 0 $. Considering that $ [L_j^t(\bs s_j^t)]_{j \in \mc N_i} \in \wh{\mc X}(y_\Ni, z_\Ni) $ for all $ \bs s_\Ni \in \mc S_\Ni $ by construction, this leads to a contradiction. 
	The proof of \eqref{eq::opInv} follows similar arguments. 	
\end{proof}

\begin{proof}[\large \bf Proof of Theorem \ref{thm::5}]
	We show that every feasible solution of \PDdATC is feasible in \PDdAT. Let $ (\wh{\bs \Gamma}_i, \wh{\mc X}_i) $ for all $ i \in \mc M $ be a feasible solution in \PDdATC. Since the state of agent $i$ evolve according to \eqref{eq::stateDynamics}, we can conclude that at time $t$ we have	 
	\begin{equation*}
	\begin{array}{r@{}l}
	x_{i,t} &=\displaystyle A_{i,1}^t x_{i,1} + \sum_{\tau=1}^{t-1} \Big(A_{i,\tau+1}^t B_{i,\tau} (Y_{\Ni,\tau} \bs s_{\Ni,\tau} + z_{\Ni,\tau}) + A_{i,\tau+1}^t D_{i,\tau} \wh{\Gamma}_{i,\tau}(\bs \xi_i^{\tau-1}, \bs s_\Ni^\tau) + A_{i,\tau+1}^t E_{i,\tau} \xi_{i,\tau} \Big) \\
	&=: \wh{\chi}_{i,t}(\bs \xi_{i}^{t-1}, \bs s_{\Ni}^{t-1})
	\end{array}
	\end{equation*}
	where where $ A_{i,\tau}^t = A_{i,\tau}A_{i,\tau+1}\dots A_{i,t-1} $ for $ \tau<t $ and $ A_{i,t}^t = I $. To show that $ \wh{\bs \Gamma}_i $ is feasible in \PDdAT, we first construct the state of agent $i$ which evolves according to $\bs x_{i} = f_{i}\big(\bs \zeta_{\Ni}, \wh{\bs \Gamma}_{i}(\bs{\xi}_i,\bs s_\Ni), \bs \xi_{i}\big)$. Starting with $\widetilde \chi_{i,1} = x_{i,1}$, we have that
	\begin{equation}\label{eq::c1_p1}
	\begin{array}{r@{}l}
	x_{i,t} &=\displaystyle A_{i,1}^t x_{i,1} + \sum_{\tau=1}^{t-1} \Big(A_{i,\tau+1}^t B_{i,\tau} \bs \zeta_{\Ni,\tau} + A_{i,\tau+1}^t D_{i,\tau} \wh{\Gamma}_{i,\tau}(\bs \xi_i^{\tau-1}, \bs s_\Ni^\tau) + A_{i,\tau+1}^t E_{i,\tau} \xi_{i,\tau} \Big) \\
	&=\displaystyle A_{i,t-1}^t x_{i,1} + \sum_{\tau=1}^{t-1} \Big(A_{i,\tau+1}^t B_{i,\tau} \bs \zeta_{\Ni} + A_{i,\tau+1}^t D_{i,\tau} \wh{\Gamma}_{i,\tau}(\bs \xi_i^{\tau-1}, [R^{\tau}_j(\bs \zeta_j^{\tau})]_{j\in \mc N_i}) + A_{i,\tau+1}^t E_{i,\tau} \xi_{i,\tau} \Big)\\
	& = \wh{\chi}_{i,t}\left(\bs \xi_i^{t-1},[R^{t-1}_j(\bs \zeta_j^{t-1})]_{j\in \mc N_i}\right)\\
	& =: \wt{\chi}_{i,t}\left(\bs \xi_i^{t-1},\bs \zeta_\Ni^{t-1}\right).
	\end{array}
	\end{equation}
	where the implications follow due to the mapping \eqref{app::map2}.
	For each $i\in \mc M$, we consider the decision $ \wt{\bs \Psi}_{i}(\cdot) \in \mc{SC}(\bs \xi_i, \bs \zeta_{\Ni}) $ defined through
	\begin{equation}\label{eq::c2_p1}
	\wt{\Psi}_{i,t}\left(\bs \xi_i^{t-1},\bs \zeta_\Ni^{t}\right)= \wh{\Gamma}_{i,t}\left(\bs \xi_i^{t-1},[R^{t}_j(\bs \zeta_j^{t})]_{j\in \mc N_i}\right).
	\end{equation}
	Notice that \eqref{eq::c2_p1} defines a valid policy construction due to the mapping \eqref{app::map2}. It remains to show that $\wh{\bs \Gamma}_{i}$ is feasible also for the constraints of \PDdAT. We do so using deduction, as follows:  
	\begin{equation}
	\begin{array}{rll}
	& \big(\wh{\bs\chi}_{i}(\bs \xi_{i},\bs s_{\Ni}), \wh{\bs \Gamma}_{i}(\bs{\xi}_i,\bs s_{\Ni})\big) \in \mc O_i, & \forall \bs \xi_i \in {\Xi}_i, \forall \bs s_\Ni \in \mc S_\Ni,\\
	\implies & \big(\wh{\bs\chi}_{i}(\bs \xi_{i},[R_j(\bs \zeta_j)]_{j \in \mc N_i}),\wh{\bs \Gamma}_{i}(\bs{\xi}_i,[R_j(\bs \zeta_j)]_{j \in \mc N_i})\big) \in \mc O_i,& \forall \bs \xi_i \in {\Xi}_i, \forall \bs \zeta_\Ni \in \wh{\mc X}_\Ni,\\
	\implies &\big(\wt{\bs \chi}_{i}(\bs \xi_i,\bs \zeta_\Ni),\wt{\bs \Psi}_{i}(\bs \xi_i, \bs \zeta_{\Ni})\big) \in \mc O_i,& \forall \bs \xi_i \in {\Xi}_i, \forall \bs \zeta_\Ni \in \wh{\mc X}_\Ni,
	\end{array}
	\end{equation}
	where the implications directly follow from \eqref{eq::c1_p1} and \eqref{eq::c2_p1}, and Lemma \ref{lem::1}. Same reasoning applies to all constraints in the problem formulation. This feasible solution attains the same value, $ \ell $, for the objective functions of \PDdATC and \PDdAT, that is:
	\begin{equation}
	\begin{array}{rl}
	\ell & = \sum\limits_{i = 1}^M \max\limits_{\bs \xi_i \in \Xi, \bs s_\Ni \in \mc S_\Ni} J_i\big(\wh{\bs\chi}_{i}(\bs \xi_{i},\bs s_{\Ni}), \wh{\bs \Gamma}_{i}(\bs{\xi}_i,\bs s_{\Ni})\big) \\& = \left \lbrace \begin{array}{l}
	J_i\big(\wh{\bs\chi}_{i}(\bs \xi_{i},\bs s_{\Ni}), \wh{\bs \Gamma}_{i}(\bs{\xi}_i,\bs s_{\Ni})\big) \leq\ell_i,~~ \forall \bs \xi_i \in {\Xi}_i, \forall \bs s_\Ni \in \mc S_\Ni\\
	\sum_{i = 1}^{M} \ell_i = \ell,
	\end{array} \right \rbrace \\
	& = \left \lbrace \begin{array}{l}
	J_i\big(\wh{\bs\chi}_{i}(\bs \xi_{i},[R_j(\bs \zeta_j)]_{j \in \mc N_i}),\wh{\bs \Gamma}_{i}(\bs{\xi}_i,[R_j(\bs \zeta_j)]_{j \in \mc N_i})\big) \leq\ell_i,~~ \forall \bs \xi_i \in {\Xi}_i, \forall \bs \zeta_\Ni \in \wh{\mc X}_\Ni\\
	\sum_{i = 1}^{M} \ell_i = \ell,
	\end{array} \right \rbrace \\
	& = \sum\limits_{i = 1}^M \max\limits_{\bs \xi_i \in {\Xi}_i, \bs \zeta_\Ni \in \wh{\mc X}_\Ni} J_i\big(\wt{\bs \chi}_{i}(\bs \xi_i,\bs \zeta_\Ni),\wt{\bs \Psi}_{i}(\bs \xi_i, \bs \zeta_{\Ni})\big) = \ell.
	\end{array}
	\end{equation}
	The implications directly follow from \eqref{eq::c1_p1} and \eqref{eq::c2_p1}, and Lemma \ref{lem::1}.
	
	Similarly, we now show that every feasible solution of \PDdAT is feasible in \PDdATC. Let $ (\wt{\bs \Psi}_i, \wh{\mc X}_i) $ for all $ i \in \mc M $ be feasible in \PDdAT. Since the state of agent $i$ evolve according to \eqref{eq::stateDynamics}, we can conclude that at time $t$ we have	 
	\begin{equation}
	\begin{array}{r@{}l}
	x_{i,t} &=\displaystyle A_{i,1}^t x_{i,1} + \sum_{\tau=1}^{t-1} \Big(A_{i,\tau+1}^t B_{i,\tau} \bs \zeta_{\Ni,\tau} + A_{i,\tau+1}^t D_{i,\tau} \wt{\Psi}_{i,\tau}(\bs \xi_i^{\tau-1}, \bs \zeta_\Ni^\tau) + A_{i,\tau+1}^t E_{i,\tau} \xi_{i,\tau} \Big)\\
	&=: \wt{\chi}_{i,t}(\bs \xi_{i}^{t-1}, \bs \zeta_{\Ni}^{t-1})
	\end{array}
	\end{equation}
	To show that $ \wt{\bs \Psi}_i $ is feasible in \PDdATC, we first construct the state of agent $i$ which evolves according to $\bs x_{i} = f_{i}\big(Y_{\Ni} \bs s_{\Ni} + z_{\Ni},\wt{\bs \Psi}_{i}(\bs{\xi}_i, \bs \zeta_{\Ni}), \bs \xi_{i}\big)$. Starting with $\widehat \chi_{i,1} = x_{i,1}$, we have that
	\begin{equation}\label{eq::c1_p2}
	\begin{array}{r@{}l}
	x_{i,t}&=\displaystyle A_{i,1}^t x_{i,1} + \sum_{\tau=1}^{t-1} \Big(A_{i,\tau+1}^t B_{i,\tau} (Y_{\Ni,\tau} \bs s_{\Ni,\tau} + z_{\Ni,\tau}) + A_{i,\tau+1}^t D_{i,\tau} \wt{\Psi}_{i,t}(\bs \xi_i^{\tau-1},\bs \zeta_{\Ni}^{\tau}) + A_{i,\tau+1}^t E_{i,\tau} \xi_{i,\tau} \Big) \\ 
	&=\displaystyle A_{i,1}^t x_{i,1} + \sum_{\tau=1}^{t-1} \Big(A_{i,\tau+1}^t B_{i,\tau} [L_{j,\tau}(\bs s_{j,\tau})]_{j\in \mc N_i} + A_{i,\tau+1}^t D_{i,\tau} \wt{\Psi}_{i,t}(\bs \xi_i^{\tau-1},[L^{\tau}_j(\bs s_j^{\tau})]_{j\in \mc N_i}) + A_{i,\tau+1}^t E_{i,\tau} \xi_{i,\tau} \Big) \\
	& = \wt{\chi}_{i,t}\left(\bs \xi_i^{t-1},[L^{t-1}_j(\bs s_j^{t-1})]_{j\in \mc N_i}\right)\\
	& =: \wh{\chi}_{i,t}\left(\bs \xi_i^{t-1},\bs s_\Ni^{t-1}\right).
	\end{array}
	\end{equation}
	where the implications follow due to the mapping \eqref{app::map1}.
	For each $i\in \mc M$, we consider the decision $ \wh{\bs \Gamma}_{i}(\cdot) \in \mc{SC}(\bs \xi_i, \bs s_{\Ni}) $ defined through
	\begin{equation}\label{eq::c2_p2}
	\wh{\Gamma}_{i,t}\left(\bs \xi_i^{t-1},\bs s_\Ni^{t}\right)= \wt{\Psi}_{i,t}\left(\bs \xi_i^{t-1},[L^{t}_j(\bs s_j^{t})]_{j\in \mc N_i}\right).
	\end{equation}
	Notice that \eqref{eq::c2_p2} defines a valid policy construction due to the mapping \eqref{app::map1}. It remains to show that $\wh{\bs \Gamma}_{i}$ is feasible also for the constraints of \PDdATC. We do so using deduction, as follows:  
	\begin{equation}
	\begin{array}{rll}
	& \big(\wt{\bs\chi}_{i}(\bs \xi_{i},\bs \zeta_{\Ni}) ,\wt{\bs \Psi}_{i}(\bs{\xi}_i,\bs \zeta_{\Ni})\big) \in \mc O_i,& \forall \bs \xi_i \in {\Xi}_i, \forall \bs \zeta_\Ni \in \wh{\mc X}_\Ni,\\
	\implies & \big(\wt{\bs\chi}_{i}(\bs \xi_{i},[L_j(\bs s_j)]_{j \in \mc N_i}), \wt{\bs \Psi}_{i}(\bs{\xi}_i,[L_j(\bs s_j)]_{j \in \mc N_i})\big) \in \mc O_i,& \forall \bs \xi_i \in {\Xi}_i, \forall \bs s_\Ni \in {\mc S}_\Ni,\\
	\implies &\big(\wh{\bs \chi}_{i}(\bs \xi_i,\bs s_\Ni), \wh{\bs \Gamma}_{i}(\bs \xi_i, \bs s_{\Ni})\big) \in \mc O_i, & \forall \bs \xi_i \in {\Xi}_i, \forall \bs s_\Ni \in {\mc S}_\Ni,
	\end{array}
	\end{equation}
	where the implications directly follow from \eqref{eq::c1_p2} and \eqref{eq::c2_p2}, and Lemma \ref{lem::1}. Same reasoning applies to all constraints in the problem formulation. This feasible solution attains the same value, $ \ell $, for the objective functions of \PDdAT and \PDdATC, that is:
	\begin{equation}
	\begin{array}{rl}
	\ell & = \sum\limits_{i = 1}^M \max\limits_{\bs \xi_i \in {\Xi}_i, \bs \zeta_\Ni \in \wh{\mc X}_\Ni} J_i\big(\wt{\bs\chi}_{i}(\bs \xi_{i},\bs \zeta_{\Ni}) ,\wt{\bs \Psi}_{i}(\bs{\xi}_i,\bs \zeta_{\Ni})\big) \\& = \left \lbrace \begin{array}{l}
	J_i\big(\wt{\bs\chi}_{i}(\bs \xi_{i},\bs \zeta_{\Ni}) ,\wt{\bs \Psi}_{i}(\bs{\xi}_i,\bs \zeta_{\Ni})\big) \leq\ell_i,~~ \forall \bs \xi_i \in {\Xi}_i, \forall \bs \zeta_\Ni \in \wh{\mc X}_\Ni,\\
	\sum_{i = 1}^{M} \ell_i = \ell,
	\end{array} \right \rbrace \\
	& = \left \lbrace \begin{array}{l}
	J_i\big(\wt{\bs\chi}_{i}(\bs \xi_{i},[L_j(\bs s_j)]_{j \in \mc N_i}), \wt{\bs \Psi}_{i}(\bs{\xi}_i,[L_j(\bs s_j)]_{j \in \mc N_i})\big) \leq\ell_i,~~ \forall \bs \xi_i \in {\Xi}_i, \forall \bs s_\Ni \in {\mc S}_\Ni,\\
	\sum_{i = 1}^{M} \ell_i = \ell,
	\end{array} \right \rbrace \\
	& = \sum\limits_{i = 1}^M \max\limits_{\bs \xi_i \in {\Xi}_i, \bs s_\Ni \in {\mc S}_\Ni} J_i\big(\wh{\bs \chi}_{i}(\bs \xi_i,\bs s_\Ni), \wh{\bs \Gamma}_{i}(\bs \xi_i, \bs s_{\Ni})\big) = \ell.
	\end{array}
	\end{equation}
	The implications directly follow from \eqref{eq::c1_p2} and \eqref{eq::c2_p2}, and Lemma \ref{lem::1}.
\end{proof}

\bibliographystyle{unsrt}
\bibliography{darivianos_abrv,Papers}

\end{document}